%% file: rangeinvar.tex
\documentclass[numbers,webpdf,imaiai]{ima-authoring-template}%

\theoremstyle{thmstyletwo}%
\newtheorem{theorem}{Theorem}
\newtheorem{example}{Example}%
\newtheorem{remark}{Remark}%

\numberwithin{equation}{section}

\usepackage{comment}
\usepackage{appendix}
\newtheorem{lemma}{Lemma}[section]
\newtheorem{corollary}{Corollary}[section]
\newtheorem{assumption}{Assumption}[section]
\newcommand\uldq{\underline{dq}}
\newcommand\uldu{\underline{du}}
\newcommand\ulda{\underline{da}}
\newcommand\uldc{\underline{dc}}

\newcommand\oneI{{(1,\ldots,I)}}
\newcommand\plusgbar{}
\newcommand\plushbar{+\bar{h}}
\newcommand\nsp{\textup{nsp}}
\newcommand\rng{\textup{rng}}

\newcommand{\Margin}[1]{}
\newcommand{\revision}[1]{#1}
\begin{document}

\DOI{DOI HERE}
\copyrightyear{2022}
\vol{00}
\pubyear{2022}
\access{Advance Access Publication Date: Day Month Year}
\appnotes{Paper}
\copyrightstatement{Published by Oxford University Press on behalf of the Institute of Mathematics and its Applications. All rights reserved.}
\firstpage{1}


\title[The range invariance paradigm for parameter identification from boundary data]{Convergence guarantees for coefficient reconstruction in PDEs from boundary measurements by variational and Newton type methods via range invariance}
\author{Barbara Kaltenbacher*
\address{\orgdiv{Department of Mathematics}, \orgname{University of Klagenfurt}, \orgaddress{\street{Universit\"atsstr. 65-67}, \postcode{9020 Klagenfurt}, 
\country{Austria}}}}

\authormark{Barbara Kaltenbacher}

\corresp[*]{Corresponding author: \href{email:barbara.kaltenbacher@aau.at}{barbara.kaltenbacher@aau.at}}

\received{Date}{0}{Year}
\revised{Date}{0}{Year}
\accepted{Date}{0}{Year}


\abstract{A key observation underlying this paper is the fact that the range invariance condition for convergence of regularization methods for nonlinear ill-posed operator equations -- such as coefficient identification in partial differential equations PDEs from boundary observations -- can often be achieved by extending the searched for parameter in the sense of allowing it to depend on additional variables. 
This clearly counteracts unique identifiability of the parameter, though. The second key idea of this paper is now to restore the original restricted dependency of the parameter by penalization. This is shown to lead to convergence of variational (Tikhonov type) and iterative (Newton type) regularization methods. We concretize the abstract convergence analysis in a framework typical of parameter identification in PDEs in a reduced and an all-at-once setting. This is further illustrated by three examples of coefficient identification from boundary observations in elliptic and 
\revision{time dependent}
PDEs.
}


\maketitle

\Margin{R1 4.}

\section{Introduction}
A proof of convergence of iterative methods for parameter identification problems in partial differential equations PDEs from boundary measurements, as relevant, e.g., in tomographic applications, has been a long-standing open problem due to the fact that the convergence analysis of these methods can only be carried out under restrictions on the nonlinearity of the forward operator that could not be verified for such PDE coefficient identification problems so far. Likewise, although at a first glance not burdened with such restrictive assumptions on the forward operator, Tikhonov regularization requires the computation of a global minimizer of a functional whose (local) convexity can only be verified under similar restricions on the nonlinearity.

The goal of this paper is to revisit one of these conditions -- range invariance of the linearized forward operator -- and show that an 
\revision{introduction of additional/artificial degrees of freedom in
the searched for parameter }
\footnote{we will shortly call this ``extension''} 
\Margin{R2 2.}
often allows for its verification. Since this counteracts unique identifiability of the parameter, we restore the original restricted dependency of the parameter by a proper penalization within the reconstruction method.

\medskip

To set the stage and explain ideas, we consider an inverse problem either in its
all-at-once formulation
\begin{equation}\label{aao}
\left.\begin{array}{ll}
A(q,u)=0 &\mbox{(model equation)}\\
Cu=y&\mbox{(observation equation)}
\end{array}\right\} 
\quad \Leftrightarrow: \ \mathbb{F}(q,u)=(0,y)^T
\end{equation}
or in its reduced formulation, with a parameter-to-state operator $S:\mathcal{D}(\mathbb{F})\to V$ defined by the first equation in 
\begin{equation}\label{red}
A(q,S(q))=0\mbox{ and }C(S(q))=y \quad \Leftrightarrow: \ \mathbf{F}(q)=y.
\end{equation}
Here $A:Q\times V\to W^*$ is the model operator, $C$ the observation operator and $Q$, $V$, $W^*$, $Y$ are Banach spaces. 
More generally, consider
\begin{equation}\label{Fxy}
F(x)=y
\end{equation}
with an operator $F:\mathcal{D}(F)(\subseteq X)\to Y$, 
mapping between Banach spaces $X$ and $Y$, given noisy data $y^\delta$ with
\begin{equation}\label{delta}
\|y-y^\delta\|_Y\leq\delta
\end{equation}
and denoting by $x^\dagger$ an exact solution.
This comprises both all-at-once \eqref{aao} with $x=(q,u)$, $F=\mathbb{F}$ and reduced \eqref{red} with $x=q$, $F=\mathbf{F}$ formulations.

Due to the characteristic ill-posedness of these problems, 
regularization has to be applied, cf., e.g., 
\cite{BakKok04, EnglHankeNeubauer:1996, Kirsch:2021, KalNeuSch08, SchusterKaltenbacherHofmannKazimierski:2012}.

To guarantee well-definedness and convergence of these methods, additional assumptions on the forward operator $F$ (that is, in the above settings, $\mathbb{F}$ or $\mathbf{F}$) need to be imposed.
These basically fall into two categories:

In the first and most common are the {\em tangential cone condition}
\cite{Scherzer:1995}
\begin{equation}\label{tangcone}
\forall x,\tilde{x}\in U \,: \ \|F(x)-F(\tilde{x})- F'(x)(x-\tilde{x})\|_Y\leq c_{\textup{tc}}
\|F(x)-F(\tilde{x})\|_Y
\end{equation}
(in a neighborhood $U$ of the exact solution)
for the convergence of Landweber iteration
\cite{HankeNeubauerScherzer:1995}
and Newton's method
\cite{IRGNMIvanov}
as well as the closely related {\em weak nonlinearity} condition from
\cite{ChaventKunisch:1996}
for local convexity of the Tikhonov functional, whose global minimizer must be computed in Tikhonov regularization.
Also in this category is the Newton-Mysovskii condition 
\begin{equation}\label{Newton-Mysovskii}
\forall x,\tilde{x}\in U \,: \ \|(F'(x)-F'(\tilde{x}))F'(x)^\dagger\|_X\leq C_{\textup{NM}} \|x-\tilde{x}\|_X,
\end{equation}
(where ${}^\dagger$ denotes the generalized inverse),
used and discussed in \cite{DeEnSc98}.
A sufficient condition for these -- and actually quite often used for their verification -- is invariance of the range of the adjoint of $F'$
\begin{equation}\label{RF}
\forall x,\tilde{x}\in U \,: \quad \rng(F'(x)^\star) = \rng(F'(\tilde{x})^\star).
\end{equation}

The second, less well-explored category of nonlinearity conditions is related to range invariance of $F'$ itself
\begin{equation}\label{FR}
\forall x,\tilde{x}\in U \,: \quad \rng(F'(x)) = \rng(F'(\tilde{x})).
\end{equation}
It contains the {\em affine covariant Lipschitz condition} 
\begin{equation}\label{affine-covariantLipschitz}
\forall x,\tilde{x}\in U \,: \ \|F'(x)^\dagger(F'(x)-F'(\tilde{x}))\|_X\leq C_{\textup{acL}} \|x-\tilde{x}\|_X,
\end{equation}
(compare to \eqref{Newton-Mysovskii})
which allows to prove convergence of Newton and quasi Newton methods, cf., e.g., \cite{diss,NewtonKaczmarz,DeEnSc98,dpapertheo,KNSbook:2008,Broyden}.

Adjoint range invariance \eqref{RF} (and alongside the tangential cone condition as well as weak nonlinearity and the Newton--Mysovskii condition) can usually only be proven to hold under a full observation assumption, that is, $C$ is just the identity or some embedding operator, which excludes practically relevant scenarios of identifying coefficients from boundary measurements.
In contrast to this, range invariance of $F'$ itself \eqref{FR} clearly follows from range invariance of the derivative $S'$ of $S$ in \eqref{red} and this is actually the way it is usually verified. This condition is independent of the observation $C$ and in particular also allows for boundary observations as relevant in tomographic imaging or nondestructive testing.
We will therefore focus on range invariance of $F'$ itself here.

By means of the Closed Graph Theorem, \eqref{FR} can be shown to be equivalent to boundedness and bounded invertibility of $F'(x)^\dagger F'(\tilde{x})$ for all $x,\tilde{x}\in U$, (see \revision{Lemma~\ref{lem:ABR} in Appendix~\ref{Appendix:Lemmas}} 
\Margin{R1 1.}
and compare to \eqref{affine-covariantLipschitz})
or, alternatively written, to 
\begin{equation*}
\forall x, \tilde{x}\in U \, \exists R(x,\tilde{x})\in L(X,X)\,: \ 
R(x,\tilde{x})^{-1}\in L(X,X) \mbox{ and }  
F'(x)=F'(\tilde{x})R(x,\tilde{x}).
\end{equation*}
In view of bounded invertibility of $R(x,\tilde{x})$, it suffices to fix $\tilde{x}$ 
here
to some $x_0\in U$
\begin{equation}\label{rangeinvar}
\exists x_0\in U\, \forall x\in U\, \exists R(x)\in L(X,X)\,: \  
R(x)^{-1}\in L(X,X) \mbox{ and }
F'(x)=F'(x_0)R(x),
\end{equation}
\revision{where sufficiency can be seen by setting $R(x,\tilde{x})=R(\tilde{x})^{-1}R(x)$.}
\Margin{R1 1.}
Here $U\subset \mathcal{D}(F)$ is a neighborhood of $x^\dagger$ and $F'$ the G\^{a}teaux derivative, which we assume to exist whenever using it.
\revision{
\begin{remark}\label{rem:equivalence_ranginvar_aao-red}
The validity of range invariance for arbitrary observation operator $C$ is formally equivalent for all-at-once \eqref{aao} and reduced \eqref{red} formulations, provided $\tfrac{\partial A}{\partial u}(q,u)$ is an isomorphism for $(q,u)\in U$ (as usually assumed in the reduced setting), see Lemma~\ref{Lem:equivalence_ranginvar_aao-red} in Appendix~\ref{Appendix:Lemmas}.
Note however, that much more freedom in the choice of function spaces is allowed in the all-at-once case and when dropping the restrictions of $\tfrac{\partial A}{\partial u}(q,u)$ to be an isomorphism and $(q,u)=(q,S(q))$, $(q_0,u_0)=(q_0,S(q_0))$. This becomes crucial when establishing $R(x)\in L(X,X),\,R(x)^{-1}\in L(X,X)$.
\end{remark}
}

\subsection{Examples}
For further motivation, we provide three classical examples of parameter identification problems where range invariance of the linearization of the forward operator can be verified.
\revision{All the computations will be completely formal; for details we refer to Appendix~\ref{Appendix:IntroEx}.
\Margin{R2 (a)}
As can be clearly seen there, in the all-at-once setting we have much more freedom in choosing the function spaces as compared to the reduced case, where we are bound to PDE solution theory in order to establish the parameter-to-state map $S$.
}

\begin{example}[identification of a potential in an elliptic PDE]\label{ex:pot_ell}
\normalfont
(For the reduced case \eqref{red}, see, e.g., \cite[Example 3.2]{dpapertheo}).
Consider identification of the spatially varying potential $q\in L^2(\Omega)$ in the elliptic boundary value problem
\begin{equation}\label{PDE_pot}
\begin{aligned}
-\Delta u +q(x)u&= f\mbox{ in }\Omega\\
\partial_\nu u&= h\mbox{ on }\partial\Omega
\end{aligned} 
\end{equation}
with given $f$, $h$,
from observations $y=Cu$ of $u$, for example measurements at the boundary of the domain $\Omega$.

With an extension $\bar{h}$ of the boundary data according to \eqref{hbar}, 
we can write the inverse problem in a reduced form \eqref{red} as $\mathbf{F}(q)=y$, or in an all-at-once form \eqref{aao} as $\mathbb{F}(q,u)=(0,y)^T$ with the respective reduced and all-at-once forward operators being defined by 
\begin{equation}\label{opeq_ex:potell}
\begin{aligned}
&\mathbf{F}(q)=CS(q)=-C(-\Delta_N+q\cdot)^{-1}[f+\bar{h}], \\
&\mathbb{F}(q,u)=(-\Delta_N u+q\cdot u -f-\bar{h},\, Cu)^T.
\end{aligned}
\end{equation}
where $\Delta_N$ denotes the Neumann Laplacian \eqref{NeumannLapl}.

The linearizations of $\mathbf{F}$ and $\mathbb{F}$ are formally given by
\begin{equation}\label{opeq_ex:potell_deriv}
\begin{aligned}
&\mathbf{F}'(q)\uldq=
-C(-\Delta_N+q\cdot)^{-1}\Bigl[\uldq\cdot S(q)\Bigr], \\ 
&\mathbb{F}'(q,u)(\uldq,\uldu)=(-\Delta_N\uldu+q\cdot\uldu+\uldq\cdot u,\, C\uldu)^T,
\end{aligned}
\end{equation}
and can easily be verified to be Fr\'{e}chet derivatives in the function space setting given in 
\eqref{VWell}, \eqref{DF}. 
For any fixed $q_0\in L^2(\Omega)$, setting
\begin{equation}\label{Rs_ell}
\begin{aligned}
\mathbf{R}(q)\uldq&= \tfrac{1}{S(q_0)\plusgbar}\cdot \bigl((-\Delta_N+q_0\cdot)\, (-\Delta_N+q\cdot)^{-1}[\uldq\cdot S(q)]\bigr)\\
&= \tfrac{1}{S(q_0)\plusgbar}\cdot \bigl(\uldq\cdot S(q)+(q-q_0)\cdot S'(q)\uldq\bigr)
\\[1ex]
\mathbb{R}(q,u)(\uldq,\uldu)&=\Bigl(\tfrac{1}{u_0\plusgbar}\bigl(\uldq\cdot u+(q-q_0)\cdot\uldu \bigr) ,\, \uldu\Bigr)
\end{aligned}
\end{equation}
we get 
\begin{equation}\label{rangeinvar_red_aao}
\mathbf{F}'(q) = \mathbf{F}'(q_0)\mathbf{R}(q), \quad 
\mathbb{F}'(q,u) = \mathbb{F}'(q_0,u_0)\mathbb{R}(q,u),
\end{equation}
where boundedness away from zero of the denominators can be guaranteed by maximum principles for $S(q_0)\plusgbar$ and simply by a proper choice for $u_0\plusgbar$.
One can also prove (see 
\eqref{RIest_ell_App}) that 
\begin{equation}\label{RIest_ell}
\|\mathbf{R}(q)-\textup{id}\|_{Q\to Q}\leq C\|q-q_0\|_{Q}, \quad 
\|\mathbb{R}(q,u)-\textup{id}\|_{Q\times V\to Q\times V}\leq C\|(q,u)-(q_0,u_0)\|_{Q\times V},
\end{equation}
so that in a sufficiently small neighborhood $U$ of $q_0$ (or of $(q_0,u_0)$, respectively), the $R$ operators are bounded and boundedly invertible, which implies range invariance of the linearizations  
\[
\forall q\in U\, : \ \rng(\mathbf{F}'(q)) = \rng(\mathbf{F}'(q_0)), \quad 
\forall (q,u)\in U\, : \ \rng(\mathbb{F}'(q,u)) = \rng(\mathbb{F}'(q_0,u_0)),
\]
no matter which observation operator $C$ is contained in the definition of $\mathbf{F}$, $\mathbb{F}$. 
Note that here, besides linearity and boundedness $C\in L(V,Y)$, we do not make any assumptions on the observation operator $C$.
This example has also been shown to satisfy the tangential cone condition in the reduced setting \eqref{red} \cite{HankeNeubauerScherzer:1995}, however only with full measurements, that is, $C$ being the embedding operator $V\to L^2(\Omega)$.

Thus, range invariance is easy to verify here, even in case of restricted measurements, as opposed to the tangential cone condition or adjoint range invariance. Still, keeping also unique identifiability of $q$ in mind, the operator $C$ should not be too restrictive either. 
For example by a straightforward dimension count, uniqueness from a single boundary observation cannot be expected to hold here and in fact the full Neumann-to-Dirichlet map (see \eqref{Lambda_pot} below) assigning to any prescribed boundary flux $\partial_\nu u$ the corresponding Dirichlet data $\textup{tr}_{\partial\Omega}u$ can be shown to yield uniqueness (see \cite{Isakov:2006} for $d=2$, \cite{Nachman:1988} for $d\geq3$).
\end{example}

\begin{example}[identification of a potential in a 
\revision{time-dependent} 
PDE] \label{ex:pot_trans}
\normalfont
We add a time dimension by considering the 
\revision{transient} 
version of Example~\ref{ex:pot_ell}
\begin{equation}\label{PDE_pot_par}
\begin{aligned}
\revision{D_t^M}u-\Delta u +q(x)u&= f\mbox{ in }\Omega\times(0,T)\\
\partial_\nu u&= h\mbox{ on }\partial\Omega\times(0,T)\\
\revision{\partial_t^m u(x,0)}&=\revision{u_m(x) \ x\in\Omega\quad m\in\{0,\ldots,m-1\}} 
\end{aligned} 
\end{equation}
with given $f$ and $h$. 
\revision{Here besides the parabolic case with $D_t^M=\partial_t$ one can consider quite general transient models, containing also wave and time fractional equations.}
\Margin{R2 19.}
\revision{In particular, we have strongly damped wave equations with $M=2$ and  
$D_t^M:=\partial_{tt} - b \Delta\partial_t$ for some $b>0$ in mind.} 

With an extension $\bar{h}$ of the initial and boundary data such that $\revision{D_t^M}\bar{h}-\Delta\bar{h}=f$ in $\Omega\times(0,T)$, $\partial_\nu\bar{h}=h$ on $\partial\Omega\times(0,T)$, 
\revision{$\partial_t^m \bar{h}(x,0)=u_m(x)$, $x\in\Omega$, $m\in\{0,\ldots,M-1\}$} 
and $\hat{u}:=u-\bar{h}$, this becomes the abstract ODE 
\begin{equation}\label{abstrODE}
\revision{D_t^M}\hat{u}(t)+(-\Delta_N +q\cdot) \hat{u}= -q\cdot\bar{h}\,, \quad \hat{u}(0)=0;
\end{equation}
the hat 
\Margin{R2 5.} 
will be skipped in the following.

For a space- and possibly also time-dependent potential $\widetilde{q}$, the operator
$T(\widetilde{q})$ defined by $(T(\widetilde{q})v)(t)= \revision{D_t^M}v(t)-\Delta_Nv(t) +\widetilde{q}(t)\cdot v(t)$, can be shown to be bounded and boundedly invertible in appropriate function spaces (cf. Appendix~\ref{Appendix:IntroEx}).
With this notation, we can write $u=S(q)=-T(q)^{-1}[q\cdot \bar{h}]$ and this formula applies both for stationary and time variable potentials $q$. 
This allows to define the reduced and all-at-once forward operators
\begin{equation}\label{opeq_ex:potpar}
\begin{aligned}
&\mathbf{F}(q)=CS(q)=-CT(q)^{-1}[q\cdot\bar{h}], \\
&\mathbb{F}(q,u)=(u'-\Delta_N u+q\cdot (u\plushbar),\, Cu)^T\,.
\end{aligned}
\end{equation}
Analogously to the above elliptic example (more or less just substituting $(-\Delta_N+q\cdot)$ by $T(q)$) we arrive at range invariance \eqref{rangeinvar_red_aao} of the linearized reduced and all-at-once forward operators with
\begin{equation}\label{Rs_par}
\begin{aligned}
&\mathbf{R}(q)\uldq= \tfrac{1}{S(q_0)\plushbar}\cdot \, T(q_0)\, T(q)^{-1}[\uldq\cdot(S(q)\plushbar)]\\
&\mathbb{R}(q,u)(\uldq,\uldu)=\Bigl(\tfrac{1}{u_0\plushbar}\bigl(\uldq\cdot(u\plushbar)+(q-q_0)\cdot\uldu \bigr) ,\, \uldu\Bigr).
\end{aligned}
\end{equation}
Also smallness of the difference between the $R$ operators and the identity can be established similarly to the elliptic example.

Note that even if $q$ depends on $x$ only, due to multiplication with the time dependent function $\tfrac{1}{S(q_0)\plushbar}$ or $\tfrac{1}{u_0\plushbar}$, the result after application of one of the $R$ operators in \eqref{Rs_par} will in general be a space and time dependent function. 
Thus we are forced to use a time dependent parameter space, e.g.
\revision{$L^2(0,T;L^2(\Omega))$.}

In this example, uniqueness of $q(x)$ from a single time trace on the boundary or final time measurements of $u$ 
\[
C = \textup{tr}_{\Gamma\times(0,T)} \ \textup{ or } \ C = \textup{tr}_{\Omega\times\{T\}} 
\]
can be shown \cite{Isakov:1991,Pierce:1979,Rundell:1987}.
However, as we have just seen, establishing range invariance requires extension of the parameter space, which leads to a loss of uniqueness. 
Therefore, we will introduce a restriction procedure by penalization, that allows to show convergence to the unique solution of the original inverse problem (that is, with the original smaller parameter space).
In this example, this can be done, e.g., by means of the functional 
$\mathcal{P}(q):=\|Pq\|_Z^2$ with $Z=L^2(0,T;L^2(\Omega))$ and 
$P(q)(t)= q(t) -\frac{1}{T}\int_0^T q(s)\, ds$, using the $L^2$ projection of $q$ on the subspace of functions that are constant in time.
\end{example}
\begin{example}[identification of a diffusion coefficient in an elliptic PDE]\label{ex:diff_ell} 
\normalfont
We now consider identification of the coefficient $a(x)$ in
\begin{equation}\label{PDE_EIT}
\begin{aligned}
-\nabla\cdot(a(x)\nabla u)&= f\mbox{ in }\Omega\\
a(x)\partial_\nu u&= h\mbox{ on }\partial\Omega.
\end{aligned} 
\end{equation}
with given $f$, $h$.
In case $f=0$ and with the physical interpretation of $u$ being an electric potential, $a$ the conductivity, and the measurements defined by all voltage-current pairs achievable by excitation at electrodes on the boundary, this is the well-known and well-investigated Calder\'{o}n problem with application in electric impedance tomography EIT.

With general measurements defined by a linear operator $C$, this can be written in a reduced \eqref{red} or in an all-at-once form \eqref{aao} with $q=a$ and 
\begin{equation}\label{opeq_ex:diffusion}
\begin{aligned}
&\mathbf{F}(a)=CS(a)=-C(-\Delta_{a,N})^{-1}[f+\bar{h}], \\
&\mathbb{F}(a,u)=(-\Delta_{a,N} u -f-\bar{h},\, Cu)^T,
\end{aligned}
\end{equation}
where
\begin{equation}\label{Delta_aN}
\langle -\Delta_{a,N} u,v\rangle_{H^1(\Omega)^*, H^1(\Omega)} := \int_\Omega a \,\nabla u\cdot\nabla v\, dx \quad \forall u,v\in H^1(\Omega)
\end{equation}
and $\bar{h}$ is defined by \eqref{hbar}.

Here, verification of \eqref{FR} (and likewise of \eqref{RF}) only works out in one space dimension. To see this, we consider the derivatives
\[
\begin{aligned}
&\mathbf{F}'(a)\ulda=
-C(-\Delta_{a,N})^{-1}\Bigl[\Delta_{\ulda,N} S(a)\Bigr], \\ 
&\mathbb{F}'(a,u)(\ulda,\uldu)=(-\Delta_{a,N}\uldu-\Delta_{\ulda,N}u,\, C\uldu)^T,
\end{aligned}
\]
and write the sufficient condition for \eqref{FR} (which is also necessary if we want it to hold for arbitrary $C$) in the reduced case as
\[
(-\Delta_{a,N})^{-1}\Bigl[\Delta_{\ulda,N} S(a)\Bigr]
= (-\Delta_{a_0,N})^{-1}\Bigl[\Delta_{R(a)\ulda,N} S(a_0)\Bigr]
\quad \forall \ulda\in Q
\]
which is equivalent to 
\[
\Delta_{R(a)\ulda,N} S(a_0) 
= (-\Delta_{a_0,N})(-\Delta_{a,N})^{-1}\Bigl[\Delta_{\ulda,N} S(a)\Bigr]
\quad \forall \ulda\in Q.
\]

This can be viewed as a transport equation
\begin{equation}\label{transport_a}
\nabla r\cdot\nabla S(a_0) + r\Delta_N S(a_0) = 
\Delta_{a-a_0,N}S'(a)\ulda+ \nabla\cdot(\ulda \nabla S(a))
\end{equation}
for $r=R(a)\ulda$. Since in general the directions of the gradients $\nabla S(a_0)$ and $\nabla S(a)$ vary in space and differ from each other, an estimate of the $L^p$ norm of the $\ulda$ term will contain a term of the form $\|\nabla\ulda\|_{L^t(\Omega)}\|\nabla S(a)\|_{L^s(\Omega)}$ (with exponents $t,\,s$ chosen according to H\"older's inequality), whereas the left hand side of \eqref{transport_a} only provides directional derivatives of $r=R(a)\ulda$ in the direction of $\nabla S(a_0)$ and thus not the full $W^{1,t}(\Omega)$ seminorm of $r$ in dimension $d>1$.  

An analogous problem occurs in the all-at-once case.
Thus we are limited to the 1-d setting, 
where this loss of regularity does not occur. Indeed, it can be easily checked that the 
choice $R(a)\ulda:=-a_0\cdot\frac{(S'(a)\ulda)_x}{(S(a_0)\revision{)}_x}$ 
yields \eqref{FR}, provided $(S(a_0)\revision{)}_x$ is bounded away from zero.
\Margin{R2 6.}

An idealized formulation of the observations in EIT is given by the Neumann-to-Dirichlet N-t-D map, defined analogously to \eqref{Lambda_pot} below; for the more realistic complete electrode model, cf. \cite{SomersaloCheneyIsaacson:1992}.
As typical for tomographic applications, these are boundary and not interior observations and so verification of the tangential cone condition \eqref{tangcone} or adjoint range invariance \eqref{RF} seems out of reach. We have seen that also range invariance in the form \eqref{FR} does not work out in higher than one space dimension. Still, since \eqref{PDE_EIT} is related to \eqref{PDE_pot} via the transform 
\begin{equation}\label{transfEIT}
u\to\sqrt{a} u, \quad q=\frac14 \frac{|\nabla a|^2}{a^2}-\frac12\frac{\Delta a}{a},
\end{equation} 
Example~\ref{ex:pot_ell} shows that \eqref{FR} indeed makes sense also in tomographic applications.
Clearly, the transform \eqref{transfEIT} requires $a$ to be sufficiently smooth, which is often not the case in applications. Therefore it would be desirable to cover the problem in its original form \eqref{PDE_EIT} that only requires an $L^\infty$ coefficient $a$.
In Example~\ref{ex:diffabs} we will show that this is in fact possible in case of simultaneous identification of a diffusion and an absorption coefficent.

\end{example}

The plan of this paper is as follows.
In Section \ref{sec:convana} we derive and analyze some regularization methods making use of the structure of range invariance \eqref{rangeinvar} (or actually a more general condition in difference form \eqref{rangeinvar_diff}. Doing so, we will follow two paradigms: The first is variational, that is, by minimizing an appropriately defined objective function as known from classical Tikhonov regularization.
The second is Newton's method or frozen versions thereof, which we regularize by stabilizing each Newton step and by early stopping.
\Margin{R1 5.}
\Margin{R2 7.}
Section \ref{sec:exclass} provides a general class of applications, typical of coefficient identification in PDEs, where the abstract convergence conditions (in particular  \eqref{rangeinvar_diff}) can be verified in a reduced 
and an all-at-once setting, respectively.
Finally, in Section \ref{sec:ex}, we provide concrete examples of coefficient identification in PDEs to which this framework applies.

\subsection*{Notation}\label{sec:notation}
$L(X,Y)$ denotes the space of bounded linear operators between the normed spaces $X$ and $Y$.\\
As already done above, we will denote the range of an operator $A:X\to Y$ by $\rng(A)\subseteq Y$; its nullspace will be abbreviated by $\nsp(A)\subseteq X$.\\
The dual of $X$ is denoted by $X^*$, the (Banach space) adjoint by $A^*\in L(Y^*,X^*)$, and in case of $X$, $Y$ being Hilbert spaces, $A^\star\in L(Y,X)$ denotes the Hilbert space adjoint.\\
A closed ball with radius $\rho>0$ and center $x_0$ is denoted by $\mathcal{B}_\rho^X(x_0)=\{x\in X\, : \, \|x-x_0\|\leq\rho\}\subseteq X$.\\
We employ the notation $W^{s,p}(\Omega)$, $H^s=W^{s,2}(\Omega)$ for Sobolev spaces over some domain $\Omega$ and $L^p_\mu(J;Z)$, $W^{s,p}_\mu(J;Z)$ for Bochner-Sobolev spaces of a variable in $J$ with values in the Banach space $Z$, with respect to the measure $\mu$, cf., e.g., \cite[Section 7.1]{Roubicek}.\\
In particular, we will use $H_\diamondsuit^M(0,T;L^2(\Omega))=\{v\in H^M(0,T;L^2(\Omega))\, : \, \partial_t^m v(0)=0, \ m\in\{0,\ldots,M\}\}$.\\
The norm of the continuous embedding $X(\Omega)\to Y(\Omega)$ for function spaces on a domain $\Omega$  is denoted by $C_{X,Y}^\Omega$.\\
For $\Gamma\subseteq\overline{\Omega}$ (the closure of a $d$-dimensional domain $\Omega$), $\Gamma$ being a $d-1$ dimensional regular manifold (e.g., $\Gamma=\partial\Omega$) and $s\geq0$, by $\textup{tr}_\Gamma:H^s(\Omega)\to H^{s-1/2}(\Gamma)$ we denote the trace operator. 

\section{Structure exploiting reconstruction methods and their convergence}\label{sec:convana}

We return to the general inverse problem \eqref{Fxy} comprising reduced and all-at-once formulations.
Rather than working with \eqref{rangeinvar} directly, we will base our analysis on a
differential range invariance condition:  
\begin{equation}\label{rangeinvar_diff}
\exists x_0\in U\,, \ K\in L(\widetilde{X},Y)\, \forall x\in U \, \exists r(x)\in \widetilde{X}\,: \  F(x)-F(x_0)=Kr(x),
\end{equation}
where we assume $K\in L(\widetilde{X},Y)$ for some auxiliary space $\widetilde{X}$.
\revision{
Analogously to Remark~\ref{rem:equivalence_ranginvar_aao-red} one can show formal equivalence of \eqref{rangeinvar_diff} for arbitrary observation operator $C$ for all-at-once \eqref{aao} and reduced \eqref{red} formulations, provided $\tfrac{\partial A}{\partial u}(q,u)$ is an isomorphism for $(q,u)\in U$. 
But again the crucial difference in its verifiability is made by the much wider choice of function spaces in the all-at-once setting.
}

Condition \eqref{rangeinvar} is sufficient for \eqref{rangeinvar_diff} with $\widetilde{X}=X$, $K=F'(x_0)$, $r(x)=\int_0^1R(x_0+\theta(x-x_0))(x-x_0)\, d\theta$, provided $U$ is convex.
On the other hand, \eqref{rangeinvar_diff} implies
\begin{equation}\label{FprimeK}
\begin{aligned}
\forall x,\tilde{x}\in U\,: \ F'(x)(\tilde{x}-x) 
&= \lim_{\epsilon\to0}\frac{1}{\epsilon}\Bigl(F(x+\epsilon(\tilde{x}-x))-F(x)\Bigr)\\
&= K\lim_{\epsilon\to0}\frac{1}{\epsilon}\Bigl(r(x+\epsilon(\tilde{x}-x))-r(x)\Bigr)
= K r'(x)(\tilde{x}-x),
\end{aligned}
\end{equation}
where we have used continuity of $K$.
Since $\tilde{x}$ and therewith the direction $\tilde{x}-x$ is arbitrary this implies
\[
\forall x\in \mbox{int}(U)\,: \ F'(x)= K r'(x) = F'(x_0)r'(x_0)^{-1}r'(x),
\]
from which \eqref{rangeinvar} follows provided $U$ is open, $r$ G\^{a}teaux differentiable and $r'(x_0)$ is boundedly invertible.

We will see in Section \ref{sec:ex} 
\Margin{R1 5.}
\Margin{R2 8.}
that, e.g., in \revision{E}xample~\ref{ex:pot_ell}, condition \eqref{rangeinvar_diff} is satisfied with easier to compute expressions as compared to \eqref{Rs_ell}, namely
$r(q)= (q-q_0)\cdot\tfrac{S(q)\plusgbar}{S(q_0)\plusgbar}$ for \eqref{red},
$r(q,u)= ((q-q_0)\cdot\tfrac{u\plusgbar}{u_0\plusgbar}, \, u-u_0)$ for \eqref{aao} and likewise for \revision{E}xample~\ref{ex:pot_trans}.
\Margin{R2 9.}

Besides being weaker than \eqref{rangeinvar} by requiring less differentiability on the forward operator $F$, condition \eqref{rangeinvar_diff} also allows for an immediate computational approach by splitting the inverse problem into an ill-posed linear and a well-posed nonlinear part as
\begin{equation*}
\begin{aligned}
&K\hat{r}=y-F(x_0) &&\textup{ in }Y\\
&r(x)=\hat{r} &&\textup{ in }\widetilde{X},
\end{aligned}
\end{equation*}
see also \cite[Remark 2.2]{DeEnSc98}.


Since establishing the range invariance condition \eqref{rangeinvar_diff} will sometimes require extension of the parameter space and this may lead to a loss of unique identifiability, we here add a penalty term that in the limit restricts reconstructions to the original parameter space, where they can (more likely) be shown to be unique. We do so by means of a penalty functional $\mathcal{P}:X\to[0,\infty]$ such that $\mathcal{P}(x^\dagger)=0$ for the exact solution $x^\dagger$. The latter implies that $\mathcal{P}$ is proper; typically, $\mathcal{P}$ will also be convex, although we do not explicitely assume this in our analysis.

The problem to be solved can therefore be rewritten as a system
\begin{equation}\label{FP}
\begin{aligned}
&K\hat{r}=y-F(x_0)&&\textup{ in }Y\\
&r(x)=\hat{r}&&\textup{ in }\widetilde{X}\\
&\mathcal{P}(x)=0&&\textup{ in }[0,\infty]
\end{aligned}
\end{equation}
for the unknowns $(\hat{r},x)\in\widetilde{X}\times U$, $U\subseteq X$.

We will consider four approaches for solving the system \eqref{FP}: A variational and three Newton type methods, 
\revision{two of which will be described in Appendix~\ref{Appendix:Newton}.}
\Margin{R 2 (b)}

\subsection{Variational reconstruction}
We minimize a regularized combination of the misfit in the first two equations in \eqref{FP} and the penalty functional $\mathcal{P}:U(\subseteq X)\to[0,\infty]$
\begin{equation}\label{var}
\begin{aligned}
&(\hat{r}_{\alpha,\beta}^\delta,x_{\alpha,\beta}^\delta)\in \mbox{argmin}_{(\hat{r},x)\in\widetilde{X}\times U} J_{\alpha,\beta}^\delta(\hat{r},x)\\
&\mbox{where } J_{\alpha,\beta}^\delta(\hat{r},x)
:=\underbrace{\|K\hat{r}+F(x_0)-y^\delta\|_Y^p+\alpha\mathcal{R}(\hat{r})}_{=:J_\alpha(\hat{r})}
+\underbrace{\beta\|r(x)-\hat{r}\|_{\widetilde{X}}^b + \mathcal{P}(x)}_{=:J_\beta(\hat{r},x)} 
\end{aligned}
\end{equation}
with some $p,b\in[1,\infty)$ and some proper 
regularization functional $\mathcal{R}:\widetilde{X}\to[0,\infty]$. 
In the particular case of exact penalization, that is $b=1$ and $\beta$ sufficiently large (but finite) so that  
$r(x)-\hat{r}=0$ is enforced, this reduces to Tikhonov regu\revision{l}arization
\Margin{R2 10.}
\begin{equation}\label{Tikh}
J_\alpha^{\delta\,\textup{Tikh}}(x):=\|F(x)-y^\delta\|_Y^p+\alpha\tilde{\mathcal{R}}(x)+\mathcal{P}(x)
\end{equation}
with $\tilde{\mathcal{R}}(x)=\mathcal{R}(r(x))$.

Existence of a (global) minimizer in \eqref{var} can be shown by the direct method of calculus of variations, see, e.g.,  \cite[Theorem 3.1]{HKPS:2007} provided topologies $\mathcal{T}$, $\tilde{\mathcal{T}}$, $\mathcal{T}_Y$ exists on $U$, $\widetilde{X}$, $Y$ such that
\begin{assumption}\label{ass1}
\begin{itemize}
\item sublevel sets of $J_{\alpha,\beta}^\delta$ are $\tilde{\mathcal{T}}\times\mathcal{T}$ compact;
\item $K$ is $\tilde{\mathcal{T}}$-to-$\mathcal{T}_Y$ continuous and $\|\cdot\|_Y$ is $\mathcal{T}_Y$ lower semicontinuous;
\item $r:U\to\widetilde{X}$ is $\mathcal{T}$ sequentially closed, that is, 
\[ \forall (x_n)_{n\in\mathbb{N}}\subseteq U\,\ \hat{r}\in \widetilde{X}: \ 
\Bigl(x_n\stackrel{\mathcal{T}}{\longrightarrow} x \mbox{ and } r(x_n)\stackrel{\widetilde{X}}{\longrightarrow} \hat{r}\Bigr)
\ \Rightarrow \ r\in U\mbox{ and }r(x)=\hat{r}; \]
\item $\mathcal{P}$ is $\mathcal{T}$ lower semicontinuous;
$\mathcal{R}$ is $\tilde{\mathcal{T}}$ lower semicontinuous;
\end{itemize}
\end{assumption}
These theoretical requirements as well as exact computation of a minimizer can partially be avoided
by just relying on existence of an infimum (which trivially holds due to boundedness of the objective function from below by zero) and considering inexact solution of the minimization problem according to 
\begin{equation}\label{var_eta}
\forall (\hat{r},x)\in\widetilde{X}\times U\, : \ J_{\alpha,\beta}^\delta(\hat{r}_{\alpha,\beta,\eta}^\delta,x_{\alpha,\beta,\eta}^\delta) \leq J_{\alpha,\beta}^\delta(\hat{r},x)+\eta
\end{equation}
with $\eta>0$ and $\eta=\eta(\delta)\to0$ as $\delta\to0$.
Comparison with $(r(x^\dagger),x^\dagger)\in \widetilde{X}\times U$ and using \eqref{delta} and $\mathcal{P}(x^\dagger)=0$ yields 
\[
\begin{aligned}
&\|K\hat{r}_{\alpha,\beta,\eta}^\delta+F(x_0)-y^\delta\|_Y^p
+\alpha\mathcal{R}(\hat{r}_{\alpha,\beta,\eta}^\delta)
+\beta\|r(x)-\hat{r}_{\alpha,\beta,\eta}^\delta\|_{\widetilde{X}}^b 
+ \mathcal{P}(x_{\alpha,\beta,\eta}^\delta)\\
&\hspace*{7cm}\leq \delta^p +\alpha\mathcal{R}(r(x^\dagger))+\eta.
\end{aligned}
\]
Assuming that $\alpha=\alpha(\delta)$, $\beta=\beta(\delta)$, $\eta=\eta(\delta)$ are chosen such that 
\begin{equation}\label{alphabetaeta}
\frac{\delta^p}{\alpha}\to0\,, \quad 
\frac{\eta}{\alpha}\to0\,, \quad 
\frac{\delta^p}{\beta}\to0\,, \quad 
\frac{\alpha}{\beta}\to0\,, \quad 
\frac{\eta}{\beta}\to0\,, \quad 
\alpha\to0\,, \quad  
\eta\to0 \quad 
\mbox{ as } \delta\to0,  
\end{equation}
we therefore obtain for $\hat{r}^\delta:=\hat{r}_{\alpha(\delta),\beta(\delta),\eta(\delta)}^\delta$, $x^\delta:=x_{\alpha(\delta),\beta(\delta),\eta(\delta)}^\delta$, 
\[
\begin{aligned}
(a)&\limsup_{\delta\to0} \mathcal{R}(\hat{r}^\delta) \leq \mathcal{R}(r(x^\dagger))\\
(b)&\limsup_{\delta\to0} \|r(x^\delta)-\hat{r}^\delta\|_{\widetilde{X}}\ =0\\
(c)&\limsup_{\delta\to0} \|K\hat{r}^\delta+F(x_0)-y^\delta\|_Y\ =0\\
(d)&\limsup_{\delta\to0} \mathcal{P}(x^\delta)\ =0
\end{aligned}
\]

This yields convergence provided topologies $\mathcal{T}$, $\tilde{\mathcal{T}}$, $\mathcal{T}_Y$ exists on $U$, $\widetilde{X}$, $Y$ such that for arbitrary sequences $(\hat{r}_n)_{n\in\mathbb{N}}\subseteq\widetilde{X}$
\begin{assumption}\label{ass2}
\begin{itemize}
\item[(i)] 
sublevel sets of $\mathcal{R}$ are $\tilde{\mathcal{T}}$ compact;
\item[(ii)] $\hat{r}_n\stackrel{\widetilde{X}}{\longrightarrow}0
\ \Rightarrow \ \exists (\hat{r}_{n_k})_{k\in\mathbb{N}}: \ \hat{r}_{n_k}\stackrel{\tilde{\mathcal{T}}}{\longrightarrow} 0$;
\item[(iii)] $K$ is $\tilde{\mathcal{T}}$-to-$\mathcal{T}_Y$ continuous and $\|\cdot\|$ is $\mathcal{T}_Y$ lower semicontinuous;
\item[(iv)] $r(x_n)\stackrel{\tilde{\mathcal{T}}}{\longrightarrow} \hat{r}
\ \Rightarrow \ \exists (x_{n_k})_{k\in\mathbb{N}}\subseteq U, \, x\in U: \ (x_{n_k}\stackrel{\mathcal{T}}{\longrightarrow} x$ and $r(x)=\hat{r})$;
\item[(v)] $\mathcal{P}$ is $\mathcal{T}$ lower semicontinuous.
\end{itemize}
\end{assumption}

\begin{theorem}\label{th:conv_var}
Under condition \eqref{rangeinvar_diff} and Assumption~\ref{ass2}, the approximations $x^\delta$ defined by \eqref{var_eta} with \eqref{alphabetaeta} converge $\mathcal{T}$ subseqentially to a solution of 
\eqref{Fxy} with $\mathcal{P}(x)=0$, 
(or equivalently, of \eqref{FP}), 
that is, every subsequence of  $(\hat{r}^\delta,x^\delta)_{\delta>0}$ has a $\tilde{\mathcal{T}}\times\mathcal{T}$ convergent subsequence and the limit $(\hat{r},x)$ of every $\tilde{\mathcal{T}}\times\mathcal{T}$ convergent subsequence solves 
\eqref{FP}. 
If the $\mathcal{R}$ minimizing solution $(\hat{r}^\dagger,x^\dagger)$ of 
\eqref{FP}\footnote{i.e., $(\hat{r}^\dagger,x^\dagger)$ solving \eqref{FP} such that 
$\forall (\hat{r},x)\mbox{ solving \eqref{FP}}\,: \  \mathcal{R}(\hat{r}^\dagger,x^\dagger)\leq \mathcal{R}(\hat{r},x)$}
is unique and $r$ is injective, then $\hat{r}^\delta\stackrel{\tilde{\mathcal{T}}}{\longrightarrow}r(x^\dagger)$ and $x^\delta\stackrel{\mathcal{T}}{\longrightarrow}x^\dagger$.
\end{theorem}
\begin{proof}
(a) and (i) imply that $\hat{r}^\delta$ has a $\tilde{\mathcal{T}}$ convergent subsequence. For any such $\tilde{\mathcal{T}}$ convergent subsequence 
$\hat{r}^{\delta_n}\stackrel{\tilde{\mathcal{T}}}{\longrightarrow}\hat{r}$, 
by (b), (ii) we also have 
$r(x^{\delta_n})\stackrel{\tilde{\mathcal{T}}}{\longrightarrow}\hat{r}$, 
which by (iv) implies existence of a sub-subsequence $(x^{\delta_{n_k}})_{k\in\mathbb{N}}\subseteq U$ and an element $x\in U$ such that 
$x^{\delta_{n_k}}\stackrel{\mathcal{T}}{\longrightarrow}x$ and $r(x)=\hat{r}$.
From (c), (iii) we conclude that $Kr(x)=y-F(x)$ and from (d), (v) $\mathcal{P}(x)=0$.
In case of uniqueness, a subsequence-subsequence argument together with injectivity of $r$ implies convergence of the whole family $(\hat{r}^\delta,x^\delta)_{\delta\in(0,\bar{\delta})}$ as $\delta\to0$.
\end{proof}
\begin{remark}\label{rem:altmin}
Computing a 
minimizer to \eqref{var} with $b>1$ is easier than finding a 
minimizer to the classical Tikhonov functional \eqref{Tikh}, since in $J_{\alpha,\beta}^\delta$ the first part $J_\alpha^\delta$ is convex due to linearity of $K$ and so is usually the function $\mathcal{P}$. 
The nonlinear term $\|r(x)-\hat{r}\|_{\widetilde{X}}^b$ is convex with respect to $\hat{r}$ and also with respect to $x$ in a neighborhood of $(\hat{r},x)=(r(x^\dagger),x^\dagger)$,\footnote{uniformly so in case of a uniformly convex norm $\|\cdot\|_{\widetilde{X}}$ and $b>1$}  
whereas the same holds for $\|F(x)-y^\delta\|_Y^p$ only under some additional restriction on the nonlinearity of the forward operator $F$ -- the already mentioned weak nonlinearity condition from \cite{ChaventKunisch:1996}, that is related to the tangential cone condition \eqref{tangcone}.
%
Thus, numerical approximation of a minimizer to \eqref{var} can, e.g, be achieved by a Gauss-Seidel (also called alternating minimization or coordinate descent) method based on alternatingly minimizing $J_\alpha(\hat{r})$ with respect to $\hat{r}$ and then, for fixed $\hat{r}$, minimizing $J_\beta(\hat{r},x)$ with respect to $x$, see, e.g.,
\cite{Auslender:1971,BertsekasTsitsiklis:1989,GrippoSciandrone:1999,OrtegaReinboldt:1970,LuoTseng:1992}.
\end{remark}
\Margin{R1 6.}

\subsection{Iterative reconstruction}\label{sec:iterative}

\revision{We exemplarily describe and analyze a simple regularized Newton method here and provide an analysis of two further (somewhat more sophisticated) Newton type methods in Appendix~\ref{Appendix:Newton}.}

\subsubsection{A frozen Newton method}\label{subsec:frozenNewton}
Based on the operator $K$ in \eqref{FP} with $r(x)\approx x-x_0$ and \eqref{rangeinvar_diff} with $\widetilde{X}=X$ we can define a frozen Newton iteration by 
\begin{equation}\label{frozenNewton}
x_{n+1}^\delta \in \mbox{argmin}_{x\in U}
\|K(x-x_n^\delta)+F(x_n^\delta)-y^\delta\|_Y^p+\alpha_n\mathcal{R}(x)+\mathcal{P}(x).
\end{equation}
We specify the setting to the case of Hilbert spaces $X$ and $Y$, assume $p=2$, $\mathcal{R}(x)=\|x-x_0\|_X^2$ and $\mathcal{P}$ to be defined by a linear operator $P\in L(X,Z)$ for another Hilbert space $Z$, $\mathcal{P}(x)=\|Px\|_Z^2$. This allows to explicitely write the minimizer as 
\begin{equation}\label{frozenNewtonHilbert}
x_{n+1}^\delta=x_{n}^\delta+(K^\star K+P^\star P+\alpha_n \textup{id})^{-1}
\Bigl(K^\star (y^\delta-F(x_n^\delta))-P^\star Px_n^\delta+\alpha_n(x_0-x_n^\delta)\Bigr),
\end{equation}
where ${}^\star $ denotes the Hilbert space adjoint.
Using \eqref{rangeinvar_diff}, which implies 
$F(x_n^\delta)=F(x_0)+Kr(x_n)=y+K(r(x_n)-r(x^\dagger))$,
and the fact that $P^\star Px^\dagger=0$, we can express the recursion for the error as
\[
\begin{aligned}
x_{n+1}^\delta-x^\dagger=&(K^\star K+P^\star P+\alpha_n \textup{id})^{-1}\\
&\Bigl(K^\star (y^\delta-y)+K^\star K\bigl((r(x^\dagger)-r(x_n^\delta))-(x^\dagger-x_n^\delta)\bigr)+\alpha_n(x_0-x^\dagger)\Bigr).
\end{aligned}
\] 
\Margin{R1 6.}
Assuming $r$ to be close to the identity (corresponding to the assumption $\|R(x)-\textup{id}\|\leq c<1$ or $\|R(x)-\textup{id}\|\leq C\|x-x^\dagger\|$ usually made in the context of \eqref{rangeinvar}, cf., e.g., \cite{DeEnSc98,Broyden,NewtonKaczmarz})
\begin{equation}\label{rid}
\exists c\in(0,1)\, \forall x\in U\, : \ \|(r(x^\dagger)-r(x))-(x^\dagger-x)\|_X\leq c\|x^\dagger-x\|_X
\end{equation}
and using spectral calculus for the selfadjoint nonnegative linear operator $A:=K^\star K+P^\star P$ we obtain
\begin{equation}\label{est_err_frozenN}
\|x_{n+1}^\delta-x^\dagger\|_X\leq \frac{\delta}{\sqrt{\alpha_n}} + c\|x_n^\delta-x^\dagger\|_X + \revision{a_n(P)}
\end{equation}
with 
\begin{equation}\label{anto0}
\revision{a_n(P)}=\alpha_n\|(K^\star K+P^\star P+\alpha_n \textup{id})^{-1}(x_0-x^\dagger)\|_X\to0\mbox{ as }n\to\infty
\end{equation}
provided $x_0-x^\dagger\in (\nsp(K)\cap \nsp(P))^\bot\subseteq \nsp(A)^\bot$,
which follows from the implication 
\[
(K^\star K+P^\star P)v=0 \ \Rightarrow \ 0=\langle v,(K^\star K+P^\star P)v\rangle_X =\|Kv\|_Y^2+\|Pv\|_Z^2.
\]
In \eqref{est_err_frozenN} we have used the following estimates, 
\revision{whose proof can be found in Appendix~\ref{Appendix:Lemmas}.}
\begin{lemma}\label{spectralbounds}
For $K\in L(X,Y)$, $P\in L(X,Z)$ with $X,Y,Z$ Hilbert spaces satisfying 
\begin{equation}\label{NKNP}
\textup{(a) }\nsp(K)^\bot \subseteq \nsp(P)
\textup{ or \ (b) }P^\star P(K^\star K)^{1/2}=(K^\star K)^{1/2}P^\star P
\end{equation} 
and any $\alpha>0$, the estimates
\begin{equation}\label{KPalpha}
\begin{aligned}
&\|(K^\star K+P^\star P+\alpha \textup{id})^{-1}K^\star K\|\leq C\\
&\|(K^\star K+P^\star P+\alpha \textup{id})^{-1}K^\star \|=\|K(K^\star K+P^\star P+\alpha \textup{id})^{-1}\|\leq\sqrt{\frac{C}{\alpha}},
\end{aligned}
\end{equation} 
hold with $C=1$ in case (a) and $C=2$ in case (b).
\end{lemma}
\begin{theorem}\label{thm:convfrozenNewton}
Let $x_0\in U:=\mathcal{B}_\rho(x^\dagger)$ for some $\rho>0$ sufficiently small and let 
$x_0-x^\dagger\in \bigl(\nsp(K)\cap \nsp(P)\bigr)^\bot$ and \eqref{NKNP} hold.
Assume that $F$ satisfies \eqref{rangeinvar_diff} with $r$ G\^{a}teaux differentiable, 
satisfying \eqref{rid}.
Let the stopping index $n_*=n_*(\delta)$ be chosen such that 
\begin{equation}\label{nstar}
n_*(\delta)\to0, \quad \delta\sum_{j=0}^{n_*(\delta)-1}c^j\alpha_{n_*(\delta)-j-1}^{-1/2} \to 0 \qquad \textup{ as }\delta\to0
\end{equation}
with $c$ as in \eqref{rid}.

Then the iterates $(x_n^\delta)_{n\in\{1,\ldots,n_*(\delta)\}}$ are well-defined by \eqref{frozenNewtonHilbert}, remain in $\mathcal{B}_\rho(x^\dagger)$ and converge in $X$, $\|x_{n_*(\delta)}^\delta-x^\dagger\|_X\to0$ as $\delta\to0$. In the noise free case $\delta=0$, $n_*(\delta)=\infty$ we have $\|x_n-x^\dagger\|_X\to0$ as $n\to\infty$.
\end{theorem}
\begin{proof}
From estimate \eqref{est_err_frozenN} we deduce the recursion
\[
\|x_{n+1}^\delta-x^\dagger\|\leq \delta\sum_{j=0}^nc^j\alpha_{n-j}^{-1/2} + c^{n+1}\|x_0^\delta-x^\dagger\|_X + \sum_{j=0}^nc^ja_{n-j},
\]
where by \eqref{anto0}, $\sum_{j=0}^nc^ja_{n-j}\to0$ as $n\to\infty$.
The choice of $n_*(\delta)$ according to \eqref{nstar} thus yields the assertions.
\end{proof}

\section{A class of examples}\label{sec:exclass}
We now discuss a framework in which range invariance \eqref{rangeinvar_diff} can be verified.
It will allow for considering multiple coefficients $q=(q^{(0)},\ldots,q^{(I)})\in Q=\prod_{i=0}^I Q^{(i)}$, $I\in\mathbb{N}_0$. This setting turns out to even enable verification of \eqref{rangeinvar_diff} in the problem of identifying a diffusion coefficient in higher space dimensions from boundary data (for which verification of range invariance and also of the tangetial cone condition so far only succeeded in one space dimension) as long as it is to be identified together with another coefficient, see Example~\ref{ex:diffabs} below.

On the other hand, we also consider multiple equations in the sense of several experiments carried out, which lead to states $u=(u^j)_{j\in J}$ with corresponding observations $C^ju^j$.

We thus study the following general model equation
\begin{equation}\label{DUBquf}
\begin{aligned}
&
(\revision{D^ju^j}+\sum_{i=0}^I B^{i,j}(u^j)q^{(i)}-f^j)_{j\in J}=0\\
&\textup{where }q=(q^{(i)})_{i\in\{0,\ldots,I\}}, \quad u:=(u^j)_{j\in J}, 
\end{aligned}
\end{equation} 
\Margin{R2 12.}
where the operators $\revision{D=\textup{diag}((D^j)_{j\in J})}\in L(V_J,W^*_J )$, $V_J\subseteq L^2_\mu(J,V)$, $W^*_J\subseteq L^2_\mu(J,W^*)$, $B^{i,j}(u^j)\in L(Q^{(i)}, W^* )$, for any $u^j\in V$, $f^j\in W^* $ are given inhomogeneities, $\mu$ is a measure on $J$, and $L^2_\mu(J,V)$, $L^2_\mu(J,W^*)$ are the corresponding Bochner-Sobolev spaces.
In here, $D$ will typically be a differential operator, possibly equipped with homogeneous boundary and/or initial conditions, and $B^{(i)}(u^j)q^{(i)}$ are terms in which the unknown parameters $q^{(i)}$ appear as coefficients or unknown source terms (the latter corresponding to $B^{(i)}(u^j)$ being constant in $u^j$).
\footnote{In principle, also the spaces $V$ and $W^*$ could depend on $j$ but we do not consider this in order not overly complicate notation.}

Note that we restrict ourselves to model operators that are (affinely) linear with respect to the coeffcients $q$, but not necessarily linear with respect to the state $u$. We will see in Section~\ref{sec:ex} below that this comprises the examples from the introduction as well as several others.

The observations are defined by the equations
\begin{equation}\label{Cuy}
C^ju^j=y^j, \quad C^j\in L(V,Y^j), \quad j\in J,
\end{equation} 
with some arbitrary operators $C^j$; other than linearity and boundedness, we will not impose further conditions on them in the following.

Clearly, more general choices also of the index set $\{0,\ldots,I\}$ for the parameters are possible in principle. While they might be explored for particular applications in future work, in this paper we stay with the setting of finitely many parameters (each of which can be an infinite dimensional quantity).

\medskip

In order to achieve range invariance, we will -- similarly to Example~\ref{ex:pot_trans} \revision{--} 
\Margin{R2 13}
introduce an artificial dependence of one of the parameters on $j$, that is, for some fixed index $i_0\in\{0,\ldots,I\}$ 
\revision{(to be chosen such that the corresponding operator $(B^{i_0,j}(u_0^j))_{j\in J}$ is an isomorphism) see \eqref{B0reg}) below}, 
\Margin{R2 15.}
we consider $q^{(i_0)}=(q^{i_0,j})_{j\in J}$. 
To recover the original independence of $q^{(i_0)}=(q^{i_0j})_{j\in J}$ on $j$ 
\revision{
we use the penalty term $\mathcal{P}(q):=\|Pq\|_{L^2_\mu(J;Q^{(i_0)})}$, where $P:L^2(J;Q^{(i_0)})\to L^2(J;Q^{(i_0)})$, 
$q\mapsto q-\frac{1}{\int_J w\,d\mu}\int_J w\, q \,d\mu $ for some weight function $w>0$, $w\in L^1_\mu(J)\cap L^2_\mu(J)$ (so that for any $v\in L^2_\mu(J)$, $\frac{1}{\int_J w\,d\mu}\int_J w\, v \,d\mu$ is the projection of $v$ on the constants in the weighted $L^2$ space with weight $w$).
}

Correspondingly, one of the parameter spaces is lifted to be the  
\revision{Bochner space $L^2_\mu(J;Q^{(i_0)})$.}
To simplify notation, we continue using the same symbol for the original parameter $q^{(i_0)}\in Q^{(i_0)}$ and its $j$ dependent extension 
$q^{(i_0)}\in 
\revision{L^2_\mu}(J;Q^{(i_0)})
=:Q^{(i_0)}_J$.

This covers some interesting cases such as time or frequency dependent PDEs with $J=(0,T)$, $J=(\underline{\omega},\overline{\omega})$ or Neumann-to-Dirichlet observations for elliptic PDEs with $J=\mathbb{N}$, see Examples~\ref{ex:pot_revisited}, \ref{ex:diffabs} below.

\medskip

In order to verify the range invariance condition \eqref{rangeinvar_diff}, throughout this section we will make the assumption that the extended version of one of the operators $B^{(i_0)}(u_0)$ is an isomorphism between $W^*_J$ and  some auxiliary space 
$\widetilde{Q}^{(i_0)}_J:=
L^2_\mu(J;\widetilde{Q}^{(i_0)})
$:
\begin{equation}\label{B0reg}
\begin{aligned}
\exists i_0\in\{0,\ldots,I\}\, :&\ B_0:=B^{(i_0)}(u_0)
:=(B^{i_0,j}(u_0^j))_{j\in J} 
\textup{ bijective and } \\
&B^{(i_0)}(u_0)\in L(\widetilde{Q}^{(i_0)}_J,W^*_J), \quad
B^{(i_0)}(u_0)^{-1}\in L(W^*_J,\widetilde{Q}^{(i_0)}_J);
\end{aligned}
\end{equation}
here $u_0:=(S(q_0))$ is the solution to \eqref{DUBqjuf} in the reduced case \eqref{red} and arbitrary in the all-at-once case \eqref{aao}.
In particular, \eqref{B0reg} can always be achieved 
by defining $\widetilde{Q}^{(i_0)}_J:=(B^{(i_0)}(u_0))^{-1}(W^*_J)$ as the preimage of $W^*_J$ under the operator $B^{(i_0)}(u_0)$ with the norm
\begin{equation}\label{Qtil}
\|\hat{r}^{(i_0)}\|_{\widetilde{Q}^{(i_0)}_J}:=\|B^{(i_0)}(u_0)\hat{r}^{(i_0)}\|_{W^*_J}
\end{equation}
Indeed, this possibly weaker norm still enables boundedness of $K=\mathbf{F}(q_0)$ or $K=\mathbb{F}(q_0,u_0)$, respectively.
By an appropriate numbering of the unknowns $q^{(i)}$, without loss of generality we assume $i_0=0$. 

\smallskip

With the extended parameter $q^{(0)}=q^{(i_0)}$, equation \eqref{DUBquf} becomes
\begin{equation}\label{DUBqjuf}
\begin{aligned}
&A(q,u):=
(\revision{D^ju^j}+B^{0,j}(u^j)q^{0,j}+\sum_{i=1}^I B^{i,j}(u^j)q^{(i)}-f^j))_{j\in J}\\
&\qquad\quad=Du+B^{(0)}(u)q^{(0)}+\sum_{i=1}^I B^{(i)}(u)q^{(i)}-f=Du+B(u)q=0
\end{aligned}
\end{equation} 
where
\[
\begin{aligned}
&B^{(0)}(u)q^{(0)}:=(B^{0,j}(u^j)q^{0,j})_{j\in J}, \quad 
B^{(i)}(u)q^{(i)}:=(B^{i,j}(u^j)q^{(i)})_{j\in J},\ i\in\{1,\ldots,I\}\\
&q=(q^{(i)})_{i\in\{0,\ldots,I\}}, \quad u:=(u^j)_{j\in J}, \quad f:=(f^j)_{j\in J}.
\end{aligned}
\]
The single experiment case is still covered by setting $J=\{0\}$, $Q^{(i)}_J=Q^{(i)}$ and the single coefficient case by setting $I=0$.

\revision{
\begin{example}[Example~\ref{ex:pot_trans} revisited]\label{ex:pot_par_revisited1} 
\normalfont
We slightly extend Example~\ref{ex:pot_trans} to the reconstruction of multiple potentials $q^{(0)},\ldots q^{(I)}$ in 
\[
D_t^M u-\Delta_N u + \sum_{i=0}^{I}q^{(i)}\cdot \Phi^{(i)}(u) = 0,
\]
where $\Phi^{(1)}$,\ldots, $\Phi^{(I)}$ are smooth real functions. 
This is clearly covered by the framework of this section with 
\[
\begin{aligned}
& j=t, \quad J=(0,T), \quad u^j=u(t) \ \in L^2(\Omega),\\
&(Du)^j=D_t^Mu(t)-\Delta_Nu(t), \quad 
B^{(i)}(u^j)q^{(i)}=q^{(i)}\cdot \Phi^{(i)}(u)\\
&J=(0,T), \quad Q^{(0)}_J= L^2(0,T;L^2(\Omega)), \quad Q^{(i)}=L^2(\Omega), \ i\in\{1,\ldots,I\}
\end{aligned}
\]
provided there exists $c>0$ such that $|\Phi^{(0)}(u_0)(x,t)|\geq c>0$ for all $(x,t)\in\Omega\times(0,T)$ so that the multiplication operator $B_0=B^{(0)}(u_0):Q^{(0)}_J\to W_J^*$, $q^{(0)}\mapsto \Phi(u_0)\cdot q^{(0)}$ is an isomorphism.
This works out, e.g., for the strongly damped wave equation example
\[
u_{tt}-\Delta_N u - b\Delta_N u_t 
+ \sum_{i=0}^{I}q^{(i)}\cdot \Phi^{(i)}(u) = 0
\]
with $b>0$ in the function space setting (cf. Appendix~\ref{Appendix:IntroEx})
\[
V=L^\infty(0,T;H^2(\Omega))\cap W^{1,\infty}(0,T;H^1(\Omega))\cap H_\diamondsuit^2(0,T;L^2(\Omega)), \quad W^*=L^2(0,T;L^2(\Omega)).
\]
\end{example}
Identification of multiple coefficients that are of much more different nature, namely a diffusion and an absorption coefficient, will be considered in Example~\ref{ex:diffabs};
see also \cite{nonlinearity_imaging_both} for another such two-coefficient identification problem.
}
\Margin{R2 15.}

\begin{remark}\label{rem:NKNP_ex}
Note that verification of condition (a) in \eqref{NKNP} and $x_0-x^\dagger\in (\nsp(K)\cap\nsp(P))^\bot$ in Theorem~\ref{thm:convfrozenNewton} (or the analogous conditions for the other Newton type methods) is related to uniqueness proofs for the particular PDE and measurement setting and therefore to be established on a case-by-case basis, see Section~\ref{sec:ex} below. We will therefore in the corollaries of the current section just re-state \eqref{NKNP} etc. as assumptions.

In fact, we will use $K=\mathbf{F}(q_0)$ in the reduced and $K=\mathbb{F}(q_0,u_0)$ \revision{ in the all-at-once setting}, 
\Margin{R1 3.} 
\Margin{R2 16.} 
which implies $r'(q_0)=\textup{Id}$ and in the examples of Section~\ref{sec:ex} show that 
\begin{equation}\label{NKNP_ex}
\nsp(F'(q_0))\cap\nsp(P)=\{0\}.
\end{equation}
With the canonical choice $u_0=S(q_0)$, this also covers the all-at-once setting due to the identity 
\begin{equation}\label{nsp_aao}
\nsp(\mathbb{F}(q_0,S(q_0)))=\{(\uldq,-(D+[B'(S(q_0)\cdot]q_0)^{-1}B(S(q_0))\uldq\, : \, \uldq\in \nsp(\mathbf{F}(q_0))\}.
\end{equation}
\end{remark}

\subsection{Reduced formulation}\label{sec:ex_red}
We use \eqref{red} in the spaces 
\[
X=Q_J= 
L^2_\mu(J,Q^{(0)})
\times \prod_{i=1}^I Q^{(i)}=Q^{(0)}_J\times\prod_{i=1}^I Q^{(i)}, \quad  
Y=Y_J\subseteq\prod_{j\in J} Y^j\,.
\]
With
\begin{equation}\label{DFred}
\begin{aligned}
q_0\in\mathcal{D}(\mathbf{F})\subseteq&\{q\in Q_J\, : \, u:=S(q)\in V_J\textup{ uniquely determined by }(D+B(\cdot)q)=f\\
&\hspace*{2cm}\textup{$B$ G\^{a}teaux differentiable at $S(q)$ and }\\
&\hspace*{2cm}D+[B'(S(q))\cdot] q:V_J\to W^*_J\mbox{ an isomorphism }\}, 
\end{aligned}
\end{equation}
we have
\begin{equation}\label{F_red}
\mathbf{F}(q)=CS(q)
\mbox{ where } S(q)=(D+B(\cdot)q)^{-1}(f),
\end{equation}
and by implicit differentiation of the identity $D+B(S(q))q=f$, 
\begin{equation}\label{Fprime_red}
\mathbf{F}'(q)=CS'(q)
\mbox{ where } S'(q)\uldq=-(D+[B'(S(q))\cdot]q)^{-1}B(S(q))\uldq.
\end{equation}
Due to \eqref{DFred} and our assumption \eqref{B0reg} of $B_0=B^{(i_0)}(S(q_0))$ being an isomorphism, we have
\begin{equation*}
\mathbf{F}'(q_0)\in L(\widetilde{Q}_J,Y_J)\ \textup{ for } 
\widetilde{Q}_J= 
L^2_\mu(J,\widetilde{Q}^{(0)})
\times \prod_{i=1}^I Q^{(i)}=\widetilde{Q}^{(0)}_J\times\prod_{i=1}^I Q^{(i)}.
\end{equation*}
A sufficient condition for the range invariance \eqref{rangeinvar_diff} with $K=\mathbf{F}'(q_0)$ to hold for an arbitrary observation operator $C$ is 
\begin{equation}\label{rangeinvarS}
\forall q\in Q_J \, \exists r(q)=(r^{0,j}(q))_{j\in J},r^{(i)}(q))_{i\in\{1,\ldots,I\}})\in \widetilde{Q}_J\,: \ S(q)-S(q_0)=S'(q_0)r(q).
\end{equation}
Due to the identities
\begin{eqnarray}
&&DS(q)+B(S(q))q=f \label{Sq}\\
&&DS(\tilde{q})+B(S(\tilde{q}))\tilde{q}=f \label{Sqtil}\\
&&D(S(q)-S(\tilde{q}))+[B(S(q))-B(S(\tilde{q}))]\tilde{q}+B(S(q))(q-\tilde{q})=0 \label{Sqdiff}\\
&&D(S'(\tilde{q})\uldq)+[B'(S(\tilde{q}))S'(\tilde{q})\uldq]\tilde{q}+B(S(\tilde{q}))\uldq=0, \label{Sqprime}
\end{eqnarray}
that imply 
\[
\begin{aligned}
&Dv+B'(S(\tilde{q}))v\tilde{q}\\ 
&+[B(S(q))-B(S(\tilde{q}))-B'(S(\tilde{q}))(S(q)-S(\tilde{q}))]\tilde{q}
+B(S(q))(q-\tilde{q})-B(S(\tilde{q}))\uldq=0
\end{aligned}
\]
for $v=S(q)-S(\tilde{q})-S'(\tilde{q})\uldq$ and any $q,\tilde{q}\in\mathcal{D}(\mathbf{F})$, $\uldq\in Q$, and the assumed bijectivity of the operator 
$D+[B'(S(\tilde{q}))\cdot]\tilde{q}$, condition \eqref{rangeinvarS} is equivalent to 
\[
\begin{aligned}
&[B(S(q))-B(S(q_0))-B'(S(q_0))(S(q)-S(q_0))]q_0\\
&+[B(S(q))-B(S(q_0))](q-q_0)-B(S(q_0))(r(q)-(q-q_0))
=0\,.
\end{aligned}
\]
Under assumption \eqref{B0reg}, this can be achieved by setting $r(q)=(r^{(0)}(q),\ldots,r^{(I)}(q))$ with
\begin{equation}\label{def_r_red}
\begin{aligned}
r^{(0)}(q):=&q^{(0)}-q^{(0)}_0 + r^{(0)}_\textup{diff}(q,S(q);q_0,S(q_0))\\
r^{(i)}(q):=&r^{(i)}(q^{(i)})=q^{(i)}-q^{(i)}_0, \quad i\in\{1,\ldots,I\}.
\end{aligned}
\end{equation}
with 
\begin{equation}\label{r0diff}
\begin{aligned}
&r^{(0)}_\textup{diff}(q,u;q_0,u_0):=B_0^{-1}\Bigl([B(u)-B(u_0)]q-[B'(u_0)(u-u_0)]q_0
\Bigr)\\
&\quad = B_0^{-1}\Bigl([B(u)-B(u_0)-B'(u_0)(u-u_0)]q_0
+[B(u)-B(u_0)](q-q_0)\Bigr).
\end{aligned}
\end{equation}

We now investigate the further convergence conditions for the methods in Section~\ref{sec:convana}. 

A verification of Assumption~\ref{ass2} for the variational approach can be done in the following setting
\begin{equation}\label{ass2_red}
\left.
\begin{aligned}
&Q^{(0)}_J, Q^{(i)} \textup{ are Banach spaces with }
\tilde{\mathcal{T}}=\mathcal{T}:=\prod_{i=0}^I \mathcal{T}^{(i)} \textup{ the norm topology on $Q_J$;}\\
&Y_J \textup{ is reflexive or dual of a separable space, }
\mathcal{T}_Y \textup{ the 
weak(*) topology;}\\
&q_0\in U:=\mathcal{B}_\sigma(q^\dagger)\subseteq\mathcal{D}(\mathbf{F}) \textup{ with $\sigma>0$ sufficiently small;}\\
&\textup{Sublevel sets of $\mathcal{R}$ are compact in $Q$ (wrt the norm topology in $Q$);}\\
&\textup{The functional $\mathcal{P}$ is $\mathcal{T}$ lower semicontinuous;}\\
&\textup{The operator $B$ is continuously differentiable on }
S(U)=\{S(q)\, : \, q\in U\}
\\
&\hspace*{1cm}\textup{with uniformly bounded dervative $B'$ on $S(U)$; }\\
&(D+[B'(S(q))\cdot]q)^{-1}B(S(q))\in L(\widetilde{Q},V),
\end{aligned}
\right\}\end{equation}
and based on Theorem~\ref{th:conv_var} yields a convergence result for variational regularization of \eqref{DUBqjuf} in a reduced setting.
\begin{corollary}[variational, reduced]\label{cor:var_red}
Under conditions \eqref{Cuy}, \eqref{B0reg}, \eqref{ass2_red}, $D\in L(V_J,W^*_J)$, $\widetilde{Q}_J=Q_J$, the approximations $x^\delta=q^\delta$ defined by \eqref{var_eta}, \eqref{F_red} with \eqref{alphabetaeta} converge norm subseqentially to a solution of 
\eqref{Fxy}, \eqref{F_red} with $\mathcal{P}(q)=0$, 
If the $\mathcal{R}$ minimizing solution $(\hat{r}^\dagger,x^\dagger)=(\hat{r}^\dagger,q^\dagger)$ of 
\eqref{Fxy}, \eqref{F_red} with $\mathcal{P}(q)=0$ 
is unique, then $\hat{r}^\delta\stackrel{\tilde{\mathcal{T}}}{\longrightarrow}r(q^\dagger)$ and $\|q^\delta-q^\dagger\|_{Q_J}\to0$ as $\delta\to0$.
\end{corollary}
\begin{proof}
Items (i), (ii), (iii), (v) in Assumption~\ref{ass2} directly follow from \eqref{Cuy}, \eqref{B0reg}, \eqref{ass2_red}, $q_0\in \mathcal{B}_\sigma(q^\dagger)\subseteq\mathcal{D}(\mathbf{F})$, and $K=CS'(q_0)$ with $S'(q_0)$ as in \eqref{Fprime_red}.

In order to verify item (iv), we assume that 
\begin{equation}\label{ri}
\hat{r}^{(i)}_n:=r^{(i)}(q_n)\stackrel{\tilde{\mathcal{T}}^{(i)}}{\longrightarrow} \hat{r}^{(i)}, 
\ i\in\{0,\ldots,I\}
\end{equation}
and have to show existence of a subsequence $(q_{n_k})_{k\in\mathbb{N}}$ such that 
$q^{(i)}_{n_k}\stackrel{\mathcal{T}^{(i)}}{\longrightarrow} q^{(i)}_*$ and $r^{(i)}(q_*)=\hat{r}^{(i)})$.
This is trivial for $i\in\{1,\ldots,I\}$ by the definition \eqref{def_r_red} of $r$ (here it would even suffice to use the weak topologies for both $\hat{r}$ and $q$) which implies 
\begin{equation}\label{qi}
q^{(i)}_n=r^{(i)}(q_n)+q^{(i)}_0\stackrel{\mathcal{T}^{(i)}}{\longrightarrow} 
\hat{r}^{(i)}+q^{(i)}_0=:q^{(i)}_*
\quad  i\in\{1,\ldots,I\}.
\end{equation}
For $i=0$, we define $q^{(0)}_n$ implicitly by  
\begin{equation}\label{Gqnr}
\begin{aligned}
G(\revision{q^{(0)}_n})=&0\ \textup{ with } G=r^{(0)},\textup{ that is, }\\ 
G(q)=&q^{(0)}-q^{(0)}_0+r^{(0)}_\textup{diff}(q,S(q);q_0,S(q_0))\\
=&q^{(0)}-q^{(0)}_0+B_0^{-1}\Bigl(
[B(S(q))-B(S(q_0))-B'(S(q_0))(S(q)-S(q_0))]q_0
\\&\qquad\qquad\qquad\qquad
+[B(S(q))-B(S(q_0))](q-q_0)\Bigr).
\end{aligned}
\end{equation}
\Margin{R2 17.}
To resolve \eqref{Gqnr} with respect to $q_n^{(0)}$ and show its convergence as a consequence of \eqref{ri} and \eqref{qi}, we use the Implicit Function Theorem.
For this purpose note that by
\begin{equation}\label{Gprime}
\begin{aligned}
\frac{\partial G}{\partial q^{(0)}}(q)\uldq^{(0)}=&\uldq^{(0)}-B_0^{-1}\Bigl\{
\bigl[\bigl(B'(S(q))-B'(S(q_0))\bigr)\frac{\partial S}{\partial q^{(0)}}(q)\uldq^{(0)}\bigr]q_0\\
&+[B'(S(q))\frac{\partial S}{\partial q^{(0)}}(q)\uldq^{(0)}](q-q_0) 
+ [B^{(0)}(S(q))-B^{(0)}(S(q_0))]\uldq^{(0)}
\Bigr\}
\end{aligned}
\end{equation}
where as in \eqref{Fprime_red}
\begin{equation}\label{dSdq0}
\frac{\partial S}{\partial q^{(0)}}(q)\uldq^{(0)}
= -\Bigl(D+[B'(S(q))\cdot] q\Bigr)^{-1}B^{(0)}(S(q))\uldq^{(0)}.
\end{equation}
Thus $G(q_0)=0$ and $\frac{\partial G}{\partial q^{(0)}}(q_0)=\textup{Id}$.
Application of the Implicit Function Theorem yields existence of neighborhoods $U_0$ of $q_0^{(0)}$ and $U'$ of $q_0^\oneI:=(q_0^{(1)},\ldots,q_0^{(I)})$ and a continuously differentiable function $\phi:U'\to U_0$ such that
\[
\begin{aligned}
&\forall q^\oneI:=(q^{(1)},\ldots,q^{(I)})\in U'\,: \ G(\phi(q^\oneI),q^\oneI)=0\\
&\forall q^{(0)}\in U_0\,, \ q^\oneI\in U'\,: \ G(q^{(0)},q^\oneI)=0 \ \Rightarrow \ q^{(0)}=\phi(q^\oneI).
\end{aligned}
\]
We can therefore define $q^{(0)}_n:=\phi(q^\oneI_n)$, $q^{(0)}_*:=\phi(q^\oneI_*)$. 
Continuity of the function $\phi$ implies the desired convergence
$q^{(0)}_n\stackrel{\mathcal{T}}{\longrightarrow} q^{(0)}_*$. 
By 
$G(q_n)=\hat{r}^{(0)}(q_n)\stackrel{\mathcal{T}}{\longrightarrow}\hat{r}^{(0)}$ and 
$G(q_n)=G(\phi(q^\oneI_n),q^\oneI_n)\stackrel{\mathcal{T}}{\longrightarrow}G(\phi(q^\oneI_*),q^\oneI_*)=r_q^{(0)}(q^{(0)}_*,q^\oneI_*)$ 
we get $r_q^{(0)}(q_*)=\hat{r}_q^{(0)}$.
Note that the Implicit Function Theorem also yields local injectivity of $r$. 
\end{proof}
\begin{remark}
Note that due to the relaxed definition \eqref{var_eta} of regularized approximations, no $\tilde{\mathcal{T}}$ lower semicontinuity of $\mathcal{R}$ is needed. This enables the choice of a stronger regularization functional $\mathcal{R}$, which in its turn enables 
the use of the strong topology in $Q$ as mandated by the use of the Implicit Function Theorem in the reduced setting (which will not be needed in the all-at-once setting).
\end{remark}

\medskip

For the frozen Newton method \eqref{frozenNewtonHilbert}, from Theorem~\ref{thm:convfrozenNewton} we conclude 
\begin{corollary}[frozen Newton, reduced]\label{cor:convfrozenNewton_red}
Let $q_0\in U:=\mathcal{B}_\rho^{Q_J}(q^\dagger)\subseteq\mathcal{D}(\mathbf{F})$ for some $\rho>0$ sufficiently small and assume that $B:V_J\to L(Q_J,W^*_J)$ is continuously differentiable, $D\in L(V_J,W^*_J)$, and $\widetilde{Q}_J=Q_J$.
Moreover, let $q_0-q^\dagger\in \bigl(\nsp(\mathbf{F}'(q_0))\cap \nsp(P)\bigr)^\bot$, \eqref{NKNP} hold with $K=\mathbf{F}'(q_0)$, \eqref{Fprime_red} and the stopping index $n_*=n_*(\delta)$ be chosen according to \eqref{nstar}.
\\
Then the iterates $(q_n^\delta)_{n\in\{1,\ldots,n_*(\delta)\}}$ are well-defined by \eqref{frozenNewtonHilbert}, \eqref{F_red}, remain in $\mathcal{B}_\rho^{Q_J}(q^\dagger)$ and converge in $Q_J$, $\|q_{n_*(\delta)}^\delta-q^\dagger\|_{Q_J}\to0$ as $\delta\to0$. In the noise free case $\delta=0$, $n_*(\delta)=\infty$ we have $\|q_n-q^\dagger\|_{Q_J}\to0$ as $n\to\infty$.
\end{corollary}
\begin{proof}
To establish \eqref{rid}, it suffices to estimate the difference between the values of $r^{(0)}_\textup{diff}(q,S(q);q_0,S(q_0))$ at $q$ and at $q^\dagger$.
Indeed, with $u_0=S(q_0)$, $u=S(q)$, $u^\dagger=S(q^\dagger)$,
\begin{equation}\label{est_rdiff}
\begin{aligned}
&r^{(0)}_\textup{diff}(q^\dagger,u^\dagger;q_0,u_0)-r^{(0)}_\textup{diff}(q,u;q_0,u_0)\\
&=B_0^{-1}\Bigl\{
\bigl[B(u)-B(u_0)-B'(u_0)(u-u_0)
-\bigl(B(u^\dagger)-B(u_0)-B'(u_0)(u^\dagger-u_0)\bigr)\bigr]q_0\\
&\hspace*{4cm}
+[B(u)-B(u_0)](q-q_0)
-[B(u^\dagger)-B(u_0)](q-q_0)\Bigr)\Bigr\}\\
&= B_0^{-1}\Bigl\{
[B(u)-B(u^\dagger)-B'(u_0)(u-u^\dagger)]q_0\\
&\hspace*{4cm}
+[B(u)-B(u^\dagger)](q-q_0)
+[B(u^\dagger)-B(u_0)](q-q^\dagger)\Bigr)\Bigr\},
\end{aligned}
\end{equation}
where according to \eqref{Sqdiff}
\begin{equation}\label{udiff}
\begin{aligned}
&\|u-u^\dagger\|_{V_J}=\|S(q)-S(q^\dagger)\|_{V_J}\\
&\leq\|(D+[B'(q^\dagger)\cdot]q^\dagger)^{-1}\|_{W^*_J\to V_J}
\|[B(u)-B(u^\dagger)-B'(u)(u-u^\dagger)]q^\dagger+B(u)(q-q^\dagger)\|_{W^*_J}.
\end{aligned}
\end{equation}
Due to continuous differentiability of $B$ we have
$[B(u)-B(u^\dagger)-B'(u)(u-u^\dagger)]q^\dagger\|_{W^*_J}= o(\|u-u^\dagger\|_{V_J})$ and thus, by \eqref{udiff}, in a sufficiently small neighborhood of $q^\dagger$, there exists a constant $M>0$ such that
$\|u-u^\dagger\|_{V_J}\leq M \|q-q^\dagger\|_{Q_J}.$
Thus, from \eqref{est_rdiff}, for $\rho$ sufficiently small we obtain \eqref{rid}.
\end{proof}

\subsection{All-at-once formulation}\label{sec:ex_aao}
Here some proofs are simpler and require less assumptions; in particular, the Implicit Function Theorem is not needed.

We use \eqref{aao} with  
\[
X=Q_J\times V_J, \quad
Y=W^*_J\times Y_J
\]
with $Q_J$, $V_J$, $W_J$, $Y_J$ as in Section~\ref{sec:ex_red} (and a slight overload of notation on $y$), as well as 
\begin{equation}\label{F_aao}
\mathbb{F}(q,u)=\left(\begin{array}{l}
Du+B(u)q-f\\
Cu
\end{array}\right).
\end{equation}
Thus we have
\begin{equation}\label{Fprime_aao}
\mathbb{F}'(q,u)(\uldq,\uldu)=\left(\begin{array}{l}
D\uldu+[B'(u)\uldu]q+B(u)\uldq\\
C\uldu
\end{array}\right),
\end{equation}
where due to 
\[
\begin{aligned}
\|\mathbb{F}'(q,u)(\uldq,\uldu)\|_{W^*_J\times V_J}^2
\leq& 4\|B(u_0)\|_{L(\widetilde{Q}_J,W^*_J)}^2\|\uldq\|_{\widetilde{Q}_J}^2\\
&+\bigl(4\|D\|_{L(V_J,W^*_J)}^2+4\|[B'(u_0)\cdot]q_0\|_{L(V_J,W^*_J)}^2
+\|C\|_{L(V_J,Y_J)}^2\bigr)\|\uldu\|_{V_J}^2
\end{aligned}
\]
we have 
\begin{equation*}
\mathbf{F}'(q_0,u_0)\in L(\widetilde{Q}_J\times V_J,W^*_J\times Y_J).
\end{equation*}
Therefore \eqref{rangeinvar_diff} with $K=\mathbb{F}'(q_0,u_0)$ and $(q^\dagger,u^\dagger)\in U\subseteq X=Q_J\times V_J$, $r(q,u)=(r_q(q,u),r_u(q,u))$ is equivalent to
\[\begin{aligned}
&D(u-u_0)+B(u)q-B(u_0)q_0=Dr_u(q,u)+[B'(u)r_u(q,u)]q+B(u)r_q(q,u)\\
&C(u-u_0)=C r_u(q,u).
\end{aligned}\]
A choice that enables this under condition \eqref{B0reg} for an arbitrary linear observation operator $C$ is
\begin{equation}\label{def_r_aao}
\begin{aligned}
r_q^{(0)}(q,u):=&q^{(0)}-q^{(0)}_0 
+r^{(0)}_\textup{diff}(q,u;q_0,u_0)\\
r_q^{(i)}(q,u):=&r^{(i)}(q^{(i)})=q^{(i)}-q^{(i)}_0, \quad i\in\{1,\ldots,I\}\\
r_u(q,u):=&r_u(u)=u-u_0
\end{aligned}
\end{equation}
with $r^{(0)}_\textup{diff}(q,u;q_0,u_0)$ defined as in \eqref{r0diff}.

To prove convergence of the all-at-once version of the variational method \eqref{var_eta} from Theorem~\ref{th:conv_var} we assume 
\begin{equation}\label{ass2_aao}
\left.
\begin{aligned}
&Q^{(0)}_J, \ Q^{(i)}, \ V_J, \ Y_J \textup{ are reflexive or duals of separable spaces;}\\
&\tilde{\mathcal{T}}:=
\mathcal{T}:=\prod_{i=0}^I \mathcal{T}^{(i)}\times\mathcal{T}_V, \ \mathcal{T}_Y \textup{ the 
weak(*) topologies;}\\
&\textup{Sublevel sets of $\mathcal{R}$ are $\tilde{\mathcal{T}}$ compact} 
\\
&\textup{The functional $\mathcal{P}$ is $\mathcal{T}$ lower semicontinuous;}\\
&\textup{The operator $B$ is differentiable on $\{u\in V\ : \ \exists q\in Q_J\, : \ (q,u)\in U$\};}\\
&\textup{For all $i\in\{1,\ldots,I\}$, the operators $u\mapsto B^{(i)}(\cdot)q_0$ 
are $\mathcal{T}_V$-to-$W^*_J$ continuous and}\\ 
&\textup{the operators $(q^{(i)},u)\mapsto B^{(i)}(u)q^{(i)}$
are $\mathcal{T}^{(i)}\times\mathcal{T}_V$-to-$W^*_J$ continuous.} 
\end{aligned}
\right\}\end{equation}
The conditions on $\mathcal{R}$ and $\mathcal{P}$ can, e.g., be satisfied by defining them by means of norms $\mathcal{R}(\hat{r}_q,\hat{r}_u)=\|\hat{r}_q^{(0)}\|_{\widetilde{Q}^{(0)}_J}^2+\sum_{i=1}^I\|\hat{r}_q^{(i)}\|_{Q^{(i)}}^2+\|\hat{r}_u\|_{V_J}^2$, 
$\mathcal{P}(x)=\|Px\|_Z^2$ for some $P\in L(X,Z)$.

\begin{corollary}[variational, all-at-once]\label{cor:var_aao}
Under conditions \eqref{Cuy}, \eqref{B0reg}, \eqref{ass2_aao}, 
$(q_0,u_0)\in U$, 
$D\in L(V_J,W^*_J)$, $\widetilde{Q}_J=Q_J$, 
the approximations $x^\delta=(q^\delta,u^\delta)$ defined by \eqref{var_eta}, \eqref{F_aao} with \eqref{alphabetaeta} converge $\mathcal{T}$ subseqentially to a solution $(\hat{r}^\dagger,q^\dagger,u^\dagger)$ of 
\eqref{Fxy}, \eqref{F_aao} with $\mathcal{P}(q,u)=0$, 
If the $\mathcal{R}$ minimizing solution $(\hat{r}^\dagger,q^\dagger,u^\dagger)$ of 
\eqref{Fxy}, \eqref{F_aao} with $\mathcal{P}(q,u)=0$ 
is unique, then $\hat{r}^\delta\stackrel{\mathcal{T}}{\longrightarrow}r(q^\dagger,u^\dagger)$ and $(q^\delta,u^\delta)\stackrel{\mathcal{T}}{\longrightarrow}(q^\dagger,u^\dagger)$.
\end{corollary}
\begin{proof}

Again, items (i), (ii), (iii), (v) in Assumption~\ref{ass2} directly follow from the assumptions made in the corollary.

In order to verify item (iv), we assume that 
\begin{equation*}
\hat{r}^{(i)}_{q,n}:=r_q^{(i)}(q_n,u_n)\stackrel{\mathcal{T}^{(i)}}{\longrightarrow} \hat{r}_q^{(i)}, \ i\in\{0,\ldots,I\},\quad
\hat{r}_{u,n}:=r_u(q_n,u_n)\stackrel{\mathcal{T}_V}{\longrightarrow} \hat{r}_u
\end{equation*}
and have to show existence of a subsequence $(q_{n_k},u_{n_k})_{k\in\mathbb{N}}$ such that 
$q^{(i)}_{n_k}\stackrel{\mathcal{T}^{(i)}}{\longrightarrow} q^{(i)}$, 
$u_{n_k}\stackrel{\mathcal{T}_V}{\longrightarrow} u$ and 
$r_q^{(i)}(q,u)=\hat{r}_q^{(i)}$, $r_u(q,u)=\hat{r}_u)$.
Again, by the definition \eqref{def_r_aao} of $r$, this is nontrivial only for $q^{(0)}$ but can easily established due to the identity
\[
\begin{aligned}
q^{(0)}_n=q^{(0)}_0+\hat{r}^{(0)}_{q,n}-r^{(0)}_\textup{diff}(q_n,u_n;q_0,u_0)
\end{aligned}
\]
since our assumptions imply $\prod_{i=1}^I\mathcal{T}^{(i)}\times\mathcal{T}_V$-to-$\mathcal{T}^{(0)}$ continuity of $r^{(0)}_\textup{diff}(\cdot,\cdot;q_0,u_0)$.
\end{proof}
\begin{corollary}[frozen Newton, all-at-once]\label{cor:convfrozenNewton_aao}
Let $(q_0,u_0)\in U:=\mathcal{B}_\rho^{Q_J\times V_J}(q^\dagger,u^\dagger)\subseteq Q_J\times V_J$
for some $\rho>0$ sufficiently small, assume that $B:V_J\to L(Q_J,W^*_J)$ is continuously differentiable, $D\in L(V_J,W^*_J)$, and $\widetilde{Q}_J=Q_J$.
Moreover, let $(q_0,u_0)-(q^\dagger,u^\dagger)\in \bigl(\nsp(\mathbb{F}'(q_0,u_0))\cap \nsp(P)\bigr)^\bot$, \eqref{NKNP} hold with $K=\mathbb{F}'(q_0,u_0)$ and the stopping index $n_*=n_*(\delta)$ be chosen according to \eqref{nstar}.
\\
Then the iterates $(q_n^\delta,u_n^\delta)_{n\in\{1,\ldots,n_*(\delta)\}}$ are well-defined by \eqref{frozenNewtonHilbert}, \eqref{F_aao}, remain in $\mathcal{B}_\rho^{Q_J\times V_J}(q^\dagger,u^\dagger)$ and converge in $Q_J\times V_J$, $\|q_{n_*(\delta)}^\delta-q^\dagger\|_{Q_J}\to0$, $\|u_{n_*(\delta)}^\delta-u^\dagger\|_{V_J}\to0$ as $\delta\to0$. In the noise free case $\delta=0$, $n_*(\delta)=\infty$ we have $\|q_n-q^\dagger\|_{Q_J}\to0$, $\|u_n-u^\dagger\|_{V_J}\to0$ as $n\to\infty$.
\end{corollary}
\begin{proof}
Here we can just use \eqref{est_rdiff} (without $u$, $u^\dagger$ depending on $q$) to verify
\[
\|r^{(0)}_\textup{diff}(q^\dagger,u^\dagger;q_0,u_0)-r^{(0)}_\textup{diff}(q,u;q_0,u_0)\|_{Q^{(0)}_J}
\leq c\Bigl(\|q-q^\dagger\|_{Q_J}^2+\|u-u^\dagger\|_{V_J}^2\Bigr)^{1/2},
\]
by continuous differentiability of $B$ and smallness of $\rho$, which together with the definition \eqref{def_r_aao} of $r$ implies \eqref{rid}.
\end{proof}

\section{Some concrete examples}\label{sec:ex}

\begin{example}[identification of a potential revisited]\label{ex:pot_revisited} 
\normalfont
Examples~\ref{ex:pot_ell} and ~\ref{ex:pot_trans} 
\[
-\Delta_N u +q\cdot u = f+\bar{h},\qquad
\revision{D_t^M}u-\Delta_N u +q\cdot(u\plushbar) = 0
\]
are covered by the framework of Section~\ref{sec:exclass} with 
\[
\begin{aligned}
& D=-\Delta_N && B(u)q=q\cdot u -f-\bar{h}&&\textup{ in the elliptic case}\\
& D=\revision{D_t^M}-\Delta_N&&B(u)q=q\cdot (u\plushbar)&&\textup{ in the 
\revision{transient}
case}
\end{aligned} 
\]
$Q:=L^2(\Omega)$, and $V$, $W^*$ as in \eqref{VWell}, \eqref{VWtrans}, respectively;
First of all we consider the single experiment case $J=0$ but later on, in order to also obtain uniqueness, will describe multiple observations with $J=\mathbb{N}$ in the elliptic 
and $J=(0,T)$ in the 
\revision{transient}
case.

We wish to recall the fact that \eqref{rangeinvar_diff} is here much easier to verify than \eqref{rangeinvar}, since instead of needing $S'(q)$ in \eqref{Rs_ell} in the reduced setting, the mapping $r$ is simply defined by pointwise multiplication 
\[
\begin{aligned}
&r(q)=\frac{S(q)\plusgbar}{S(q_0)\plusgbar}\cdot(q-q_0) 
\textup{ in the reduced setting},\\ 
&r_q(q,u)=\frac{u\plusgbar}{u_0\plusgbar}\cdot(q-q_0)\textup{ in the all-at-once setting.} 
\end{aligned}
\]
(where in the 
\revision{transient}
case we have to add $\bar{h}$ in the numerator and denominator).
As mentioned above, boundedness away from zero of the denominators can be taken care of in the elliptic and 
\revision{transient}
case, e.g., by maximum principles and a proper choice of the boundary conditions.

Alternatively, division by $S(q_0)\plusgbar$ or $u_0\plusgbar$ can be avoided by redefining the norm in $\widetilde{Q}$ as a weighted $L^2$ norm $\|\hat{r}\|_{\tilde{Q}}:=
\|\hat{r}\cdot u_0\|_{L^2(\Omega)}$ while using the ordinary $L^2$ norm on $Q$, i.e., 
$\|q\|_{Q}:=  \|q\|_{L^2(\Omega)}$, see \eqref{Qtil}. 
However, the requirement of assuming $S(q_0)$ or $u_0$ to be bounded away from zero returns when proving differentiability of $B$ as an operator $V_J\to L(\widetilde{Q}_J,W^*_J)$.


A proper choice of $P$ for verifying (a) in \eqref{NKNP} heavily depends on uniqueness arguments and therefore on the type of observations, cf. Remark~\ref{rem:NKNP_ex}.
We will now comment on this with the following settings known from the literature, as well as the corresponding extension sets $J$
\begin{itemize}
\item elliptic case: Neumann-to-Dirichlet (N-t-D) map, $J=\mathbb{N}$; 
\item transient 
case: time trace measurements, $J=(0,T)$; 
\end{itemize}
We target our investigations at the two frozen Newton methods, where by Corollaries~\ref{cor:convfrozenNewton_red}, \ref{cor:convNewton0_red}, \ref{cor:convfrozenNewton_aao},  \ref{cor:convNewton0_aao}, 
and due to $r'(q_0)=\textup{Id}$ we only have to verify the relatively simple nullspace condition \eqref{NKNP_ex}.
With the penalty operators defined by 
\revision{
\[
\begin{aligned}
(Pq)(j)=&q(j)-\frac{\sum_{\ell\in\mathbb{N}} \ell^{-2}q(\ell)}{\sum_{\ell\in\mathbb{N}} \ell^{-2}}\mbox{ for }J=\mathbb{N}\\
(Pq)(j)=&q(j)-\frac{\int_0^T (1+\ell)^{-2}q(\ell)\, d\ell}{\int_0^T (1+\ell)^{-2} \, d\ell}\mbox{ for }J=(0,T)
\end{aligned}
\]
(possibly including $T=\infty$),
}
verification of \eqref{NKNP_ex} amounts to proving linearized uniqueness of a potential depending only on $x$, which can be done by well-known arguments that we here shortly recall for the convenience of the reader.

\paragraph{Elliptic case.}
The observations are defined by boundary observations of $u$ according to the N-t-D map $\Lambda\in L(H^{-1/2}(\partial\Omega),H^{1/2}(\partial\Omega))$, defined by 
\begin{equation}\label{Lambda_pot}
\Lambda h =\textup{tr}_{\partial\Omega}u 
\textup{ where $u$ solves $-\Delta u+q(x)\cdot u=f$ with $\partial_\nu u=h$ on $\partial\Omega$}
\end{equation}
Considering the N-t-D map as an operator 
$\Lambda\in L(H^{-1/2}(\partial\Omega),L^2(\partial\Omega))$ 
(in view of the fact that only values and not derivatives can be measured)
it is compact and thus can be represented by means of the sequence of the images of a complete orthonormal sequence $(\varphi^j)_{j\in\mathbb{N}}$ in the separable Hilbert space $H^{-1/2}(\partial\Omega)$, 
\[
g^j=\textup{tr}_{\partial\Omega}u^j 
\textup{ where $u^j$ solves $-\Delta u+q(x)\cdot u=f$ with $\partial_\nu u^j=\varphi^j$ on $\partial\Omega$}\,, \quad j\in\mathbb{N}.
\]
The sequence $(\varphi^j)_{j\in\mathbb{N}}$ can be obtained from a complete (in $H^{-1/2}(\partial\Omega)$) sequence of given excitations $(h^j)_{j\in\mathbb{N}}$ by means of Gram-Schmidt orthonormalization (skipping linearly dependent elements in the sequence $(h^j)_{j\in\mathbb{N}}$).

Using the weak form of the PDE, we get the model equations 
\begin{equation}\label{Aj_ell_pot}
\langle A^j(q,u^j),w\rangle_{W^*,W}
:=\int_\Omega \Bigl(\nabla u^j\cdot\nabla w+qu^jw\Bigr)\, dx
-\int_{\partial\Omega}\varphi^j w \,d\Gamma(x) \,, w\in W:=H^1(\Omega)
\end{equation}
for $u^j\in V:=H^1(\Omega)$, and the observation equations
\begin{equation}\label{Cj_diffabs}
C^j(u^j)=\textup{tr}_{\partial\Omega}u^j.
\end{equation}
Thus we can use the formalism from Section \ref{sec:exclass} with 
$I=0$, $q^{(0)}=q$, $J=\mathbb{N}$, $D^j=-\Delta_N$, 
$\langle B^{0,j}(u^j)q,w\rangle_{W^*,W}=\int_\Omega q\, u^j \, w\, dx$.

Assuming $q\in\nsp(P)$ means that the 
\revision{artificial dependency of $q$ on $j$}
\Margin{R2 18.}
is formally undone and to verify \eqref{NKNP_ex}, we can proceed by proving that an only space dependent $q$ lying in $\nsp(\mathbf{F}'(q_0))$ must vanish.
(Due to the identity \eqref{nsp_aao} this also covers the all-at-once case.)
\revision{For the details we refer to Section~\ref{sec_LinUni_pot_rev_ell} in Appendix~\ref{Appendix:LinUniqueness}.}

\paragraph{Transient case.}
The proof of \eqref{NKNP_ex} we choose to present here relies on taking the Laplace transform (to this end, we assume to have observations on the whole positive time line $T=\infty$) and applies to a general transient model with 
\[
Du =\sum_{m=1}^M \mathcal{M}_m* \partial_t^m u+\sum_{n=0}^N \mathcal{N}_n*\mathcal{A}^{\beta_n} \partial_t^n u
\]
where $\mathcal{A}$ is a selfadjoint (with respect to a possible weighted $L^2$ inner product $L^2_w$ on $\Omega)$), nonnegative definite operator with domain $V\subseteq L^\infty(\Omega)$, eigensystem $(\lambda_i,\varphi_i^k)_{i\in\mathbb{N},k\in K^i}$ and spectral powers $\mathcal{A}^{\beta_n}$ defined by
$\mathcal{A}^\beta v=\sum_{i\in\mathbb{N}}\lambda_i^\beta \sum_{k\in K^i} \langle v,\varphi_i^k\rangle_{L^2_w(\Omega)}\varphi_i^k$, 
the star $*$ denotes convolution $(v*w)(t):=\int_0^tv(t-s)w(s)\, ds$,
and the kernel functions $\mathcal{M}_m$, $\mathcal{N}_n$ are assumed to have Laplace transforms $\widehat{\mathcal{M}}_m$, $\widehat{\mathcal{N}}_n$   
lying in $L^\infty(\mathbb{C})$.
Besides the parabolic and hyperbolic cases with $N=0$, $\mathcal{N}_0=\delta$ (the delta distribution centered a zero with $\widehat{\mathcal{N}}_0\equiv1$), $\beta_0=1$, and either $M=1$, $\mathcal{M}_1=\delta$ or $M=2$, $\mathcal{M}_1=0$, $\mathcal{M}_2=\delta$ respectively, this also covers many viscoelastic models, see, e.g., \cite{Atanackovic:2014,BagleyTorvik:1986,frac_TUM,oparnica2020well,saedpanah2014well} and fractional damping models in acoustics see, e.g., \cite{Holm:2019,fracPAT,fracJMGT,Szabo:1994}.
\\
A uniqueness proof only needing a finite time interval but more specific to the parabolic or hyperbolic setting could be carried out along the lines of proofs in, e.g., \cite{Isakov:2006} and the citing literature.
\\ 
We set $J=(0,\infty)$, $V_J\subseteq \{v\in C([0,\infty);H):\frac{\partial^\ell v}{\partial t^\ell}(t=0)=0\}$ for some Hilbert space $H\supseteq V$ (enforcing homogeneous initial conditions; inhomogeneous ones cound be incorporated into the right hand side $f$ and the offet $\bar{h}$). Moreover we assume that $u_0+\bar{h}$ is a space-time separable function $(u_0+\bar{h})(x,t)=\phi(x)\psi(t)$; in the reduced case where $u_0=S(q_0)$ this can be achieved by an appropriate choice of the excitation  $f$, $\bar{h}$, see, e.g., \cite{nonlinearity_imaging_fracWest}.
The observation operator is supposed to be time invariant, bounded and linear $(Cu)(t)=C_0 u(t)$ with $C_0\in L(H,Y)$.
\revision{In this setting, the intersection of the nullspaces of $\mathbb{F}'(q_0,u_0)$ and $P$ can be shown to be trivial, see Section~\ref{sec_LinUni_pot_rev_trans} in Appendix~\ref{Appendix:LinUniqueness}, which implies \eqref{NKNP_ex}.}


\begin{remark}\label{rem:finaltimeobs}
Instead of time trace observations, final observations in all of $\Omega$ for a fixed time instance would actually be simpler, allowing to set $J=\{0\}$ and employ an easy proof of linearized uniqueness; also, the inverse problem of recovering $q(x)$ is known to be less ill-posed than with time trace data, cf. e.g., \cite{Cannon:1975}.   
Since our main motivation was to allow for iterative or variational reconstruction from boundary measurements, we do not go into further detail about this here.
\end{remark} 
\end{example}

\begin{example}[combined diffusion and absorption identification]\label{ex:diffabs}
\normalfont
Identify $a(x)$ and $c(x)$ (that is, $q=(a,c)$) in 
\begin{equation}\label{PDE_diffabs}
-\nabla\cdot(a\nabla u)+cu=0\textup{ in }\Omega  
\end{equation}
in a domain satisfying \eqref{Omega} from boundary observations of $u$. 
This problem arises, e.g., in steady-state diffuse optical tomography, see, e.g., \cite{ArridgeSchotland:2009,GibsonHebdenArridge:2005,Harrach:2009};
\revision{ for an analysis of a finite element approxiation of this problem we refer to \cite{Quyen2020}.}
\Margin{R1 8.}

To give an idea on uniqueness for this problem, we quote some identifiability results from the literature:
\begin{itemize}
\item uniqueness of $a$ and $c$ from N-t-D maps $\Lambda_\lambda\in L(H^{-1/2}(\partial\Omega),H^{1/2}(\partial\Omega))$ for $-\nabla\cdot(a\nabla u)+(c-\lambda)u=0$, for all $\lambda\geq\lambda_*$, and some fixed $\lambda_*\in\mathbb{R}$, provided $a\in W^{2,\infty}(\Omega)$, $\revision{c}\in L^\infty(\Omega)$ 
\Margin{R1 7.}
and $\nabla a$ is known on $\partial\Omega$); 
see \cite[proof of Corollary 1.7]{CanutoKavian:2004}\footnote{he same result would also provide uniqueness from spectral data, but this cannot written by means of a linear operator $C$};
\item uniqueness of effective absorption $\tilde{c}:=\frac{\Delta\sqrt{a}}{\sqrt{a}}+\frac{c}{a}$ as well as the jumps of $a$ and $\partial_\nu a$ on the discontinuity set from a single N-t-D map $\Lambda=\Lambda_0\in L(L^2(\partial\Omega),L^2(\partial\Omega))$ for piecewise analytic $a$ and $c$;
see Theorem 2.2 in \cite{Harrach:2012}; a counterexample on uniqueness for general $a$, $c$ can be found in \cite{ArridgeLionheart:1998};
\item uniqueness of $a$ or of $c$ alone from a single N-t-D map $\Lambda=\Lambda_0\in L(H^{-1/2}(\partial\Omega),H^{1/2}(\partial\Omega))$, 
\begin{itemize}
\item for $c\in L^p(\Omega)$, $p>\frac{d}{2}$ if $d\geq3$, $p>1$ if $d=2$; 
see 
\cite{Isakov:2006} 
for $d=2$, 
\cite{Nachman:1988} for $d\geq3$; 
\item for $a\in C^2(\Omega)$ if $d\geq3$ \cite{SylvesterUhlmann:1987},
for $a\in L^\infty(\Omega)$ if $d=2$ \cite{AstalaPaeivarinta:2006}.
\end{itemize}
\end{itemize}

To recover $a$ and $c$ simultaneously, we therefore consider the N-t-D maps $\Lambda_\lambda\in L(H^{-1/2}(\partial\Omega),H^{1/2}(\partial\Omega))$ for all $\lambda\geq0$.
As in in the elliptic case of Example~\ref{ex:pot_revisited}, we think of the N-t-D map as represented by means of the sequence of the images of a basis $(\varphi^n)_{n\in\mathbb{N}}$ of $H^{-1/2}(\partial\Omega)$, 
\begin{equation}\label{Lambda_diffabs}
\begin{aligned}
&g^{\lambda,n}=\textup{tr}_{\partial\Omega}u^{\lambda,n}\\ 
&\textup{where $u^{\lambda,n}$ solves $-\nabla\cdot(a\nabla u)+(c-\lambda)u=0$ with $\partial_\nu u^n=\varphi^n$ on $\partial\Omega$}\,, \quad n\in\mathbb{N}, \quad \lambda\geq0.
\end{aligned}
\end{equation}

Using the weak form of \eqref{PDE_diffabs}, we get the model equations 
\begin{equation*}
\begin{aligned}
&\langle A^{\lambda,n}(c,a,u^{\lambda,n}),w\rangle_{W^*,W}
:=\int_\Omega \Bigl(a\nabla u^{\lambda,n}\cdot\nabla w+(c-\lambda)u^{\lambda,n}w\Bigr)\, dx
-\int_{\partial\Omega}\varphi^n w \,d\Gamma(x) \,,\\ 
&w\in W:=H^1(\Omega)
\end{aligned}
\end{equation*}
for $u^{\lambda,n}\in V:=H^1(\Omega)$, and the observation equations
\begin{equation*}
C^{\lambda,n}(u^{\lambda,n})=\textup{tr}_{\partial\Omega}u^{\lambda,n}.
\end{equation*}
Thus we can use the formulation from Section \ref{sec:exclass} with 
$I=1$, $q^{(0)}=c$, $q^{(1)}=a$, $J=[0,\infty)\times\mathbb{N}$, 
$D^{\lambda,n}=-\lambda\cdot$, 
$\langle B^{0,\lambda,n}(u^{\lambda,n})c,w\rangle_{W^*,W}=\int_\Omega c\,u^{\lambda,n}\, w\, dx$, 
$\langle B^{1,\lambda,n}(u^{\lambda,n})a,w\rangle_{W^*,W}=\int_\Omega a \nabla u^{\lambda,n} \cdot\nabla w\, dx$.

It is readily checked that the range invariance condition \eqref{rangeinvar_diff} holds with $K=\mathbf{F}'(c_0,a_0)$ (or $K=\mathbb{F}'(c_0,a_0,u_0)$) and 
\[
r(c,a) = 
(c-c_0
-\frac{1}{u_0^{\lambda,n}}
\Bigl(-\nabla\cdot((a-a_0)\nabla(u^{\lambda,n}-u_0^{\lambda,n}))+(c-c_0)\cdot(u^{\lambda,n}-u_0^{\lambda,n})\Bigr)
,a-a_0) 
\]
(or $r(c,a,u^{\lambda,n})=(r(c,a),u^{\lambda,n}-u^{\lambda,n}_0)$)
where in the reduced case we have to use $u^{\lambda,n}:=S^{\lambda,n}(c,a)$ as in \eqref{Lambda_diffabs}, cf. \eqref{def_r_red}, \eqref{def_r_aao}. 
Extension only affects the absorption coefficient $c(x)$, which becomes $c(x,\lambda,n)$, while the diffusion coefficient $a$ remains only $x$ dependent.
\\
We use the spaces $Q_J=
L^2(J;L^2(\Omega))
\times (L^\infty(\Omega)\cap W^{1,2P/(P-2)}(\Omega)$ with $P\in(2,\infty)$, $1-\frac{d}{2}\geq-\frac{d}{P}$, $V=H^2(\Omega)$, $W^*=L^2(\Omega)$, $Y=L^2(\partial\Omega)$.
\\
With the penalty operator defined by 
\revision{
\[
\begin{aligned}
(Pc)(\lambda,n)=c(\lambda,n)-\frac{\sum_{\ell\in\mathbb{N}} \ell^{-2}\int_{(0,\infty)} (1+\kappa)^{-2} c(\kappa,\ell)\, d\kappa}{\sum_{\ell\in\mathbb{N}} \ell^{-2}\int_{(0,\infty)} (1+\kappa)^{-2} \, d\kappa},
\end{aligned}
\]
}
verification of \eqref{NKNP_ex} amounts to proving linearized uniqueness of a diffusion coefficient and a potential, both depending only on $x$, which can be done similarly to the elliptic case of recovering a potential,
\revision{see Section~\ref{sec_LinUni_diffabs} in Appendix~\ref{Appendix:LinUniqueness}.}
\end{example}

\begin{example}[Reconstruction of a boundary coefficient]\label{ex:bndy} 
\normalfont
As a final example in which \revision{ we }
\Margin{R2 21.}
actually have uniqueness without needing to extend the coefficient we consider the problem of recovering the Robin coefficient $q$ in the elliptic boundary value problem
\begin{equation}\label{PDE_bndy}
\begin{aligned}
-\Delta u &= \ell\mbox{ in }\Omega\\
\partial_\nu u+ q\cdot\Phi(u)&= h\mbox{ on }\Gamma_R\subseteq\partial\Omega\\
\partial_\nu u&= h\mbox{ on }\Gamma_N\subseteq\partial\Omega\setminus \Gamma_R\\
u&= 0\mbox{ on }\Gamma_D:=\partial\Omega\setminus(\Gamma_R\cup\Gamma_N)
\end{aligned} 
\end{equation}
from observations $y=Cu$ of $u$, for example Dirichlet data on some part of the boundary $\partial\Omega$ of $\Omega$ or on some other surface $\Sigma$ contained in the closure of the PDE domain $\Omega$. Since we will not appeal to elliptic regularity, it suffices to have a Lipschitz domain $\Omega\subseteq\mathbb{R}^d$, $d\in\{1,2,3\}$ here.

This has applications for example in corrosion detection, and heat conduction, see, e.g., 
\cite{ChaabaneJaoua:1999,Inglese:1997}, where the latter application also motivates the use of a potentially nonlinear term $\Phi(u)$ in the boundary condition. The mentioned papers and the references therein also provide uniqueness of $q$ from a single boundary observation.

The weak form of \eqref{PDE_bndy} reads as 
\begin{equation}\label{PDE_bndy_weak}
\begin{aligned}
&u\in V:=H_D^1(\Omega)\textup{ and }\forall v\in W:=H_D^1(\Omega)\\
&\langle A(q,u),v\rangle_{W^*,W}:=\int_\Omega \nabla u \cdot \nabla v\, dx + \int_{\Gamma_R} q \cdot \Phi(u)\, v\, d\Gamma(x)
- \int_\Omega \ell v\, dx- \int_{\Gamma_R\cup\Gamma_N} h\, v\, d\Gamma(x)=0
\end{aligned} 
\end{equation}
where $H_D^1(\Omega)=\{v\in H^1(\Omega)\, : \, \textup{tr}_{\Gamma_D}v=0\}$.
Here $\Phi:L^4(\Omega)\to L^4(\Omega)$ is supposed to be a known nonlinear operator, for example a superposition (or Nemitzky) operator induced by some function $\phi:\mathbb{R}\to \mathbb{R}$ with at most linear growth at zero and infinity.
 
This problem is therefore covered by the framework of Section~\ref{sec:exclass} with 
\[
\begin{aligned}
&\langle Du,v\rangle_{W^*,W} = \int_\Omega \nabla u \cdot \nabla v\, dx, \quad
\langle B(u)q,v\rangle_{W^*,W} = \int_{\Gamma_R} q\, \Phi(u)\, v\, d\Gamma(x), \quad\\
&\langle f,v\rangle_{W^*,W} = \int_\Omega \ell v\, dx+\int_{\Gamma_R\cup\Gamma_N} h\, v\, d\Gamma(x).
\end{aligned} 
\]

We set $K=\mathbf{F}'(q_0)$ (or $K=\mathbb{F}'(q_0,u_0)$) and according to \eqref{def_r_red} (or \eqref{def_r_aao}) set
\[
r(q) = (q-q_0+\frac{1}{\Phi(u_0)}\Bigl((\Phi(u)-\Phi(u_0))q-\Phi'(u_0)(u-u_0)q_0\Bigr)
\]
(or $r(q,u)=(r(q),u-u_0))$),
where in the reduced case $u$ has to solve \eqref{PDE_bndy} and likewise for $u_0$ with $q_0$ in place of $q$.

In the reduced setting we have to take care of well-posedness of the forward problem by restricting the domain to (similarly to \eqref{DF}) 
\begin{equation*}
\begin{aligned}
\mathcal{D}(\mathbf{F}):=\{q\in L^2(\partial\Omega): &\exists \bar{q}\in L^\infty(\Omega), 
\bar{q}\geq0
\textup{ a.e. and } 
C_{\textup{PF}} L \Bigl(\|\textup{tr}_{\Gamma_R}\|_{H_D^1,H^{1/2}} C_{H^{1/2}, L^4}^{\Gamma_R}\Bigr)^2 \|q-\bar{q}\|_{L^2(\partial\Omega)}<1\}.
\end{aligned}
\end{equation*}
with $C_{\textup{PF}}$, $L$ as in \eqref{PoincareFriedrichs}, \eqref{growth} below.
Here for simplicity we have assumed that the $d-1$ dimensional measure $\textup{meas}^{d-1}(\Gamma_D)>0$ so that a Poincar\'{e}-Friedrichs equation 
\begin{equation}\label{PoincareFriedrichs}
\forall v\in H_D^1(\Omega)\, : \ \|v\|_{H^1(\Omega)}\leq C_{\textup{PF}} \|v\|_{H^1_D(\Omega)}
\end{equation}
holds.
Moreover we impose the following monotonicity and continuity conditions on $\Phi$
\begin{equation}\label{growth}
\begin{aligned}
&\forall v,w\in L^4(\Gamma_R)\,\stackrel{a.e.}{\forall} x\in\Omega : \ 
(\Phi(v)-\Phi(w))\cdot (v-w)\geq0 
\textup{ and }|\Phi(v)-\Phi(w)|\leq L |v-w|\\
&\Phi:L^4(\Gamma_R)\to L^4(\Gamma_R)^* \textup{ hemicontinuous.}
\end{aligned} 
\end{equation}
(Note that $\Phi(u)$ is supposed to act as a pointwise multiplication but $\Phi$ does not necessarily need to be a Nemitzky operator, thus might as well be nonlocal.)
From these properties we can conclude (strict, due to the $D$ term) monotonicity, coercivity and hemicontinuity of the operator $D+B(\cdot)\bar{q}$ for any nonegative $\bar{q}\in L^\infty(\Omega)$. 
By the Browder-Minty Theorem on the separable Banach space $H_D^1(\Omega)$ this implies existence and uniqueness (and by the elliptic $D$ part also continuity) of 
$S(\bar{q}):=(D+B(\cdot)\bar{q})^{-1}(\tilde{f})$ for any $\tilde{f}\in H_D^1(\Omega)^*$.
Defining the fixed point map $\mathcal{M}:u\mapsto u_+:=(D+B(\cdot)\bar{q})^{-1}(f-B(u)(q-\bar{q}))$ on $H_D^1(\Omega)$ and proving its contractivity on a sufficiently small neighborhood of $\bar{q}$ yields well-definedness of $S(q)=(D+B(\cdot)q)^{-1}(f)$.
Indeed,  by testing the PDE for the difference
\[
D(u_{1+}-u_{2+})+(B(u_{1+})-B(u_{2+}))\bar{q}=
(B(u_1)-B(u_2))(\bar{q}-q)
\]
with $u_{1+}-u_{2+}$ and using \eqref{PoincareFriedrichs}, \eqref{growth} we obtain
\[
\begin{aligned}
&\frac{1}{C_{\textup{PF}}} \|u_{1+}-u_{2+}\|_{H^1(\Omega)}^2 
\leq \langle D(u_{1+}-u_{2+})+(B(u_{1+})-B(u_{2+}))\bar{q}, u_{1+}-u_{2+}\rangle_{W^*,W}\\
&= \langle (B(u_1)-B(u_2))(\bar{q}-q), u_{1+}-u_{2+}\rangle_{W^*,W}
=\int_{\Gamma_R} (\bar{q}-q)(\Phi(u_1)-\Phi(u_2)) (u_{1+}-u_{2+})\, d\Gamma\\
&\leq L \Bigl(\|\textup{tr}_{\Gamma_R}\|_{H_D^1,H^{1/2}} C_{H^{1/2}, L^4}^{\Gamma_R}\Bigr)^2 \|q-\bar{q}\|_{L^2(\partial\Omega)}\|u_{1}-u_{2}\|_{H^1(\Omega)} \|u_{1+}-u_{2+}\|_{H^1(\Omega)}.
\end{aligned}
\]

If additionally to \eqref{growth}, $\Phi:L^4(\Gamma_R)\to L^4(\Gamma_R)$ is continuously differentiable with $\Phi'(S(q_0))$ bijective and $\Phi'(S(q_0))^{-1}\in L(L^4(\Gamma_R), L^4(\Gamma_R))$, 
then in the function space setting
\[
Q=L^2(\Gamma_R), \quad V=W=H_D^1(\Omega)
\]
with $Y$ such that $C\in L(V,Y)$ the conditions of Corollaries~\ref{cor:var_red}, \ref{cor:convfrozenNewton_red}, \ref{cor:convNewton0_red} are satisfied with, a proper choice of $\mathcal{R}$ (e.g., $\mathcal{R}(q)=H^r(\Gamma_R)$ for some $r>0$) and $\mathcal{P}$. 
Indeed, 
we can set $\mathcal{P}=0$ in case of boundary observations $C=\textup{tr}_\Gamma$ for some open subset $\Gamma\subseteq\Gamma_N$, 
and verify triviality of $\nsp(\mathbf{F}'(q_0))$, where $\mathbf{F}'(q_0)$ is the mapping that takes $\uldq$ to $Cv$ with $v$ solving
\begin{equation*}
\begin{aligned}
-\Delta v &= 0\mbox{ in }\Omega\\
\partial_\nu v+ q_0\cdot\Phi'(u_0)v&= -\uldq\cdot\Phi(u_0)\mbox{ on }\Gamma_R\\
\partial_\nu v&= 0\mbox{ on }\Gamma_N\\
v&= 0\mbox{ on }\Gamma_D.
\end{aligned} 
\end{equation*}
Thus if $Cv=\textup{tr}_\Gamma v=0$ then $v$ is harmonic in $\Omega$ with vanishing Cauchy data on $\Gamma$ and thus, due to Holmgren's Uniqueness Theorem, $v=0$ in $\Omega$.
From this we conclude that $F'(q_0)$ has trivial nullspace and therefore conditions \eqref{NKNP}, \eqref{NKNPx}, \eqref{NKNP0} as well as $q_0-q^\dagger\in \bigl(\nsp(\mathbf{F}'(q_0))\cap \nsp(P)\bigr)^\bot$, and $r(q^\dagger)\in \bigl(\nsp(\mathbf{F}'(q_0))\cap \nsp(Pr'(q_0)^{-1})\bigr)^\bot$
are satisfied with the trivial choice $P=0$.

Likewise we can apply Corollaries~\ref{cor:var_aao}, \ref{cor:convfrozenNewton_aao}, \ref{cor:convNewton0_aao} in the all-at-once setting even without having to assume \eqref{PoincareFriedrichs} and \eqref{growth}.
\end{example} 

Further examples satisfying a range invariance condition can be found, e.g., in \cite[Section 2.2]{NewtonKaczmarz}, \cite[Section 4]{diss}, and \cite{frac_space-potential,nonlinearity_imaging_both}.
\Margin{R2 22.}

\section*{Conclusion and Outlook}\label{sec:outlook}
While the use of range invariance conditions for proving convergence of regularization methods is known in the literature, their use has so far largely been limited to relatively simple model problems -- as is the case for other restrictions on the nonlinearity of the forward operator (such as the tangential cone condition) used in the analysis of solution methods for ill-posed problems.
In this paper, we suggest to overcome this limitation by possibly augmenting the dimensions of variability of the searched for parameter, thus allowing to establish validity of this condition for a much larger class of inverse problems, which indeed contains practically relevant examples.
The loss of uniqueness, which also impacts convergence of these methods (since the initial guess needs to be orthogonal to the nullspace of the linearized forward operator) is overcome by adding a penalty that in the limit enforces the original lower dimensional variability.

The analysis is carried out in a general framework (Section~\ref{sec:convana}) that covers both reduced and all-at-once formulations of inverse problems as demonstrated in Section~\ref{sec:exclass} and contains some practically relevant inverse problems of reconstruction from boundary observations as detailed in Section~\ref{sec:ex}. 

While -- for clarity of exposition -- we stayed with norm based misfit functionals in most of this paper, extension to more general functional should be considered in view of their practical usefulness for, e.g., modeling different types of data noise.

Also convergence rates under source conditions will be subject of future research.
\section*{Acknowledgment}
This work was partially supported by the Austrian Science Fund {\sc fwf}
under the grants P30054 and DOC78.

\renewcommand\appendixname{Appendix}
\begin{appendices}
\input{AppendixIntroEx}

\input{AppendixLemmas}

\input{AppendixNewton}

\input{AppendixLinUniqueness}

\end{appendices}

\end{document}

%% file: AppendixIntroEx.tex
\section{Details on examples from introduction} \label{Appendix:IntroEx}
\setcounter{example}{0}
\begin{example}[identification of a potential in an elliptic PDE]\label{ex:pot_ell}
\normalfont
Consider identification of the spatially varying potential $q\in L^2(\Omega)$ in the elliptic boundary value problem \eqref{PDE_pot} with $f\in L^2(\Omega)$, $h\in H^{1/2}(\partial\Omega)$, $\partial_\nu$ the normal derivative, from observations $y=Cu$ of $u$, for example measurements at the boundary of the domain $\Omega$, on which we here and in most of the further examples assume
\Margin{R2 3.}
\begin{equation}\label{Omega}
\Omega\subseteq\mathbb{R}^d, \ d\in\{1,2,3\}, \quad
\partial\Omega\in C^{1,1} \mbox{ or $\Omega$ polygonal and convex}
\end{equation}
With $q$ sufficiently close to a function that is positive and bounded away from zero (see \eqref{DF} below) we can avoid the problem of a nontrivial nullspace of the negative Neumann Laplacian, defined by its weak form $-\Delta_N:H^2(\Omega)\to L^2(\Omega)$
\begin{equation}\label{NeumannLapl}
\langle -\Delta_N u,v\rangle_{H^1(\Omega)^*, H^1(\Omega)} := \int_\Omega \nabla u\cdot\nabla v\, dx \quad \forall u,v\in H^1(\Omega).
\end{equation}
With this and $\bar{h}$ defined by 
\begin{equation}\label{hbar}
\langle \bar{h},v\rangle_{H^1(\Omega)^*, H^1(\Omega)} := \int_{\partial\Omega} h v\, d\Gamma(x) \quad \forall v\in H^1(\Omega).
\end{equation}
we can rewrite \eqref{PDE_pot} as an operator equation in $L^2(\Omega)$
\begin{equation}\label{pot_ell}
-\Delta_N u +q\cdot u = f+\bar{h},
\end{equation}
where $q\cdot:H^2(\Omega)\to L^2(\Omega)$ denotes the multiplication operator, which is bounded for $q\in L^2(\Omega)$, due to continuity of the embedding $H^2(\Omega)\to L^\infty(\Omega)$ (with constant denoted by $C_{H^2,L^\infty}^\Omega$ below). 

We can write the inverse problem in a reduced form \eqref{red} as $\mathbf{F}(q)=y$, or in an all-at-once form \eqref{aao} as $\mathbb{F}(q,u)=(0,y)^T$ with the respective reduced and all-at-once forward operators being defined by \eqref{opeq_ex:potell}.
With  
\begin{equation}\label{VWell}
Q:=L^2(\Omega),\quad V=H^2(\Omega),\quad W^*= L^2(\Omega),
\end{equation}
and $Y$ such that $C\in L(V,Y)$, the reduced forward operator $\mathbf{F}:Q\to Y$ is well-defined on 
\begin{equation}\label{DF}
\begin{aligned}
\mathcal{D}(\mathbf{F}):=\{q\in L^2(\Omega): &\exists \bar{q}\in L^\infty(\Omega), \gamma\in\mathbb{R}^+\, : \  \bar{q} \geq \gamma \textup{ a.e. and } \\
&\|(-\Delta_N+\bar{q}\cdot)^{-1}\|_{L^2\to H^2} C_{H^2, L^\infty}^\Omega \|q-\bar{q}\|_{L^2(\Omega)}<1\}
\end{aligned}
\end{equation}
due to the fact that by elliptic regularity, the operator $(-\Delta_N+\bar{q}\cdot):H^2(\Omega)\to L^2(\Omega)$ is boundedly invertible.
\\
In the all-at-once setting we have more freedom to choose the spaces. Fixing again $Q=L^2(\Omega)$, we may use any $V,W^*,Y$ such that $-\Delta_N+q\cdot\in L(V,W^*)$ and $C\in L(V,Y)$ holds.
For simplicity, we again take $V=H^2(\Omega)$ and correspondingly choose $W^*=L^2(\Omega)$.

In this function space setting, the linearizations of $\mathbf{F}$ and $\mathbb{F}$ given by \eqref{opeq_ex:potell_deriv}
can easily be verified to be Fr\'{e}chet derivatives. 
For any fixed $q_0\in L^2(\Omega)$, with $\mathbf{R}(q)$ according to \eqref{Rs_ell}, the range invariance \eqref{rangeinvar_red_aao}
is satisfied.
To show that the $R$ operators are close to the identity, using  
\[
\begin{aligned}
&(-\Delta_N+q_0\cdot)(-\Delta_N+q\cdot)^{-1} -\textup{id}=(q-q_0)\cdot(-\Delta_N+q\cdot)^{-1}\\ 
&S(q)-S(q_0)=-(-\Delta_N+q_0\cdot)^{-1}[(q-q_0)\cdot S(q)\revision{]}
\end{aligned}
\]
\Margin{R2 4.}
and the estimates
\[
\begin{aligned}
&\|\mathbf{R}(q)\uldq-\uldq\|_Q \\
&=
\|\tfrac{1}{S(q_0)\plusgbar}\Bigl(
\Bigl\{(-\Delta_N+q_0\cdot)(-\Delta_N+q\cdot)^{-1} -\textup{id}\Bigr\}[\uldq\cdot S(q)]
+\uldq\cdot(S(q)-S(q_0))\|_{L^2(\Omega)}\\
&\leq 
\|\tfrac{1}{S(q_0)\plusgbar}\|_{L^\infty(\Omega)}
\|(q_0-q)\cdot\bigl((-\Delta_N+q\cdot)^{-1}[\uldq\cdot S(q)]\bigr)\\
&\hspace*{4cm}-\uldq\cdot \bigl((-\Delta_N+q_0\cdot)^{-1}[(q-q_0)\cdot S(q)\bigr)\|_{L^2(\Omega)}\\
&\leq 
\|\tfrac{1}{S(q_0)\plusgbar}\|_{L^\infty(\Omega)}\|q_0-q\|_{L^2(\Omega)}\|\uldq\|_{L^2(\Omega)}
\|S(q)\|_{L^\infty(\Omega)}C_{H^2,L^\infty}^\Omega\\
&\hspace*{2cm}\Bigl(\|(-\Delta_N+q\cdot)^{-1}\|_{L^2\to H^2}+\|(-\Delta_N+q_0\cdot)^{-1}\|_{L^2\to H^2}\Bigr)
 \end{aligned}
\]
\begin{equation}\label{RIest_ell_App}
\begin{aligned}
&\|\mathbb{R}(q,u)(\uldq,\uldu)-(\uldq,\uldu)\|_{Q\times V} \\
&=\|\tfrac{1}{u_0\plusgbar}\cdot\bigl(\uldq\cdot u+(q-q_0)\cdot\uldu \bigr)-\uldq\|_Q\\
&=\|\tfrac{1}{u_0\plusgbar}\cdot\bigl(\uldq\cdot(u-u_0))+(q-q_0)\cdot\uldu \bigr)\|_Q\\
&\leq \|\tfrac{1}{u_0\plusgbar}\|_{L^\infty(\Omega)}C_{H^2,L^\infty}^\Omega
\Bigl(\|\uldq\|_{L^2(\Omega)}\|u-u_0\|_{H^2(\Omega)}+
\|q-q_0\|_{L^2(\Omega)}\|\uldu\|_{H^2(\Omega)}\Bigr),
 \end{aligned}
\end{equation}
we obtain \eqref{RIest_ell}.
\end{example}

\begin{example}[identification of a potential in a 
\revision{time-dependent}
PDE] \label{ex:pot_trans}
\normalfont
Consider identification of $q=q(x)$ in the initial boundary value problem \eqref{PDE_pot_par} with $f\in L^2(0,T;L^2(\Omega))$, $h\in L^2(0,T;H^{1/2}(\Omega))$, and $\Omega$ satisfying \eqref{Omega}. 
Using an extension $\bar{h}$ of the initial and boundary data such that $\bar{h}_t-\Delta\bar{h}=f$ in $\Omega\times(0,T)$, $\partial_\nu\bar{h}=h$ on $\partial\Omega\times(0,T)$, $\bar{h}(x,0)=u_0(x)$, $x\in\Omega$ and $\hat{u}:=u-\bar{h}$, we can write the forward problem as \eqref{abstrODE}.


With the operator
\[
T(\widetilde{q}):\, V\to W^*, \quad 
v\mapsto \Bigl(t\mapsto \revision{D_t^M}v(t)-\Delta_Nv(t) +\widetilde{q}(t)\cdot v(t)\Bigr)
\]
for $\widetilde{q}\in 
\revision{L^2}(0,T;\mathcal{D}(F))$
we can define the reduced and all-at-once forward operators by \eqref{opeq_ex:potpar}
on $\tilde{\mathcal{D}}(\mathbf{F}):=
\revision{L^2}(0,T;\mathcal{D}(\mathbf{F}))
\subseteq X=Q=
\revision{L^2}(0,T;L^2(\Omega))
$ with $\mathcal{D}(\mathbf{F})$ defined as in \eqref{DF}) and 
$\tilde{\mathcal{D}}(\mathbb{F}):=\tilde{\mathcal{D}}(\mathbf{F})\times V$, respectively. 

To prove that for any for $\widetilde{q}\in L^2(0,T;\mathcal{D}(F))$, the operator 
$T(\widetilde{q}):\, V\to W^*$ is indeed an isomorphism, the function spaces $V$ and $W^*$ have to be properly chosen.
For example, in the case of a strongly damped wave equation
\[
D_t^M u -\Delta_N u + q u = u_{tt}-\Delta_N u - b\Delta_N u_t + q u  
\]
with $b>0$ this holds true with the choice 
\begin{equation}\label{VWtrans}
\begin{aligned}
&V=L^\infty(0,T;H^2(\Omega))\cap W^{1,\infty}(0,T;H^1(\Omega))\cap H_\diamondsuit^2(0,T;L^2(\Omega)), \\ &W^*=L^2(0,T;L^2(\Omega)), \\
&H_\diamondsuit^2(0,T;L^2(\Omega))=\{v\in H^2(0,T;L^2(\Omega))\, : \, v(0)=0, v_t(0)=0\}.
\end{aligned}
\end{equation}
cf. e.g. \cite{KalLas:2009}.
Since $V\subseteq L^\infty(0,T;L^\infty(\Omega)$, we can invoke the elementary estimate  
\[
\|q\cdot u\|_{L^2(0,T;L^2(\Omega))} \leq 
\|q\|_{L^2(0,T;L^2(\Omega))} \|u\|_{L^\infty(0,T;L^\infty(\Omega))}.
\]
for establishing smallness of the difference between the $R$ operators and the identity similarly to \eqref{RIest_ell_App}.
We mention in passing that this does not work in the seemingly simpler parabolic case $M=1$, $D_t^M=\partial_t$, where a typical function space setting would be 
\begin{equation}\label{VWpar}
\begin{aligned}
&V=L^2(0,T;H^2(\Omega))\cap H_\diamondsuit^1(0,T;L^2(\Omega)), \\ &W^*=L^2(0,T;L^2(\Omega)), \\
&H_\diamondsuit^1(0,T;L^2(\Omega))=\{v\in H^1(0,T;L^2(\Omega))\, : \, v(0)=0\}.
\end{aligned}
\end{equation}
Here the use of a smooth cutoff function $\Phi$ and replacing the term $q\cdot u$ by $q\cdot\Phi(u)$ still enables an analysis, which come with additional technicalities, though.

\end{example}
\begin{example}[identification of a diffusion coefficient in an elliptic PDE]\label{ex:diff_ell} 
\normalfont
Consider identification of the coefficient $a(x)$ in the elliptic boundary value problem \eqref{PDE_EIT}
with $f\in H^1(\Omega)^*$, $h\in H^{-1/2}(\partial\Omega)$, and $\Omega$ satisfying \eqref{Omega} (as a matter of fact, a Lipschitz domain would suffice here, since we are not employing elliptic regularity in this example).

The forward operators in the reduced \eqref{red} and in the all-at-once formulation \eqref{aao} are given by \eqref{opeq_ex:diffusion} with $\Delta_{a,N}$ defined by \eqref{Delta_aN}, where $\bar{h}$ is defined by \eqref{hbar}.
In the reduced case, in order to achieve unique solvability of the Neumann problem, we assume $f\in L^1(\Omega)$, $h\in L^1(\partial\Omega)$, $\int_\Omega f(x)\, dx+\int_{\partial\Omega} h(x)\, d\Gamma(x)=0$, replace $H^1(\Omega)$ by 
$H_{\circ}^1(\Omega):=\{v\in H^1(\Omega)\, : \, \int_\Omega v(x)\, dx=0\}$ and define the domain of the forward operator as 
\begin{equation}\label{DF_diff_red}
\mathcal{D}(\mathbf{F}):=\{a\in L^\infty(\Omega)\, : \, \exists \gamma\in\mathbb{R}^+ \, : \
 a\geq\gamma\textup{ a.e.}\}.
\end{equation}
With $Q:=L^\infty(\Omega)$, $V=W=H_{\circ}^1(\Omega)$ and $Y$ such that $C\in L(V,Y)$, the reduced forward operator $\mathbf{F}:Q\to Y$ is well-defined on $\mathcal{D}(\mathbf{F})$.
For simplicity, we again use the same spaces (but skip the restriction to $\mathcal{D}(\mathbf{F})$) for defining the all-at-once forward operator $\mathbb{F}$.
\end{example}

%% file: AppendixLemmas.tex
\section{Proofs of some auxiliary results} \label{Appendix:Lemmas}
\begin{lemma}[\cite{Kirsch:2017}, extended to a Banach space setting]\label{lem:ABR}
For $X,Y$ Banach spaces, $A,B\in L(X,Y)$ the following equivalence holds
\[
\rng(A)=\rng(B) \ \Longleftrightarrow \ \exists R\in L(X,X): \ R^{-1}\in L(X,X)\mbox{ and }A=BR
\]
\end{lemma}
\begin{proof}
The implication from right to left is trivial. To show the other direction, we assume that $\rng(A)=\rng(B)$ and first of all define $R:X\to X/_{\nsp(A)}$  by assigning to any $x\in X$ the unique element $z:=Rx$ from the quotient space $X/_{\nsp(A)}$ (which due to the fact that $\nsp(A)$ is closed for $A\in L(X,Y)$, is a Banach space as well) such that $Az=Bx$. 
To prove that the resulting linear operator $R$ is bounded, we invoke the Closed Graph Theorem. The graph of $R$ is given by 
\[
\text{Gr}(R)=\{(x,Rx)\in X\times X/_{\nsp(A)}\, : \, x\in X\} =\{(x,z)\in X\times X/_{\nsp(A)}\, : \, Az = Bx\}.
\]
For an arbitrary sequence $(x_n,z_n)_{n\in\mathbb{N}}\subseteq \text{Gr}(R)$ with $x_n\to x$, $z_n\to z$ in $X$ we have, by continuity of $A$, $B$, that 
\[
\|Az-Bx\| \leq \|Az-Az_n\|+\|Bx_n-Bx\|\to0 \text{ as }n\to\infty
\]
and thus $(x,z)\in \text{Gr}(R)$. This implies that $R$ is bounded.
Boundedness of its inverse follows by exchanging the roles of $A$ and $B$.
\end{proof}

\begin{lemma}\label{Lem:equivalence_ranginvar_aao-red}
Consider \eqref{aao} and \eqref{red} under the assumption that $\tfrac{\partial A}{\partial u}(q,u)$ is an isomorphism for $(q,u)\in U$. 
Then \eqref{rangeinvar} hold for \eqref{aao} and all $C\in L$ if and only if it holds \eqref{rangeinvar} holds for \eqref{red}.
\end{lemma}

\begin{proof}
This can be seen from the chain of equivalences
\[
\begin{aligned}
&\forall C\in L(V,Y)\ : \quad 
\mathbb{F}'(q,u) = \mathbb{F}'(q_0,u_0)\mathbb{R}(q,u)\\
&\Leftrightarrow \ (\uldq,\uldu) \in Q\times V\ : \\
&\hspace*{1.5cm}\begin{cases}
\tfrac{\partial A}{\partial q}(q,u)\uldq + \tfrac{\partial A}{\partial u}(q,u)\uldu 
\\
= \tfrac{\partial A}{\partial q}(q_0,u_0)(\mathbb{R}_{qq}(q,u)\uldq + \mathbb{R}_{qu}(q,u)\uldu) + \tfrac{\partial A}{\partial u}(q_0,u_0)
(\mathbb{R}_{uq}(q,u)\uldq + \mathbb{R}_{uu}(q,u)\uldu)
\\
C\uldu = C (\mathbb{R}_{uq}(q,u)\uldq + \mathbb{R}_{uu}(q,u)\uldu)
\end{cases}\\
&\Leftrightarrow  \quad
\begin{cases}
\tfrac{\partial A}{\partial q}(q,u) 
= \tfrac{\partial A}{\partial q}(q_0,u_0)\mathbb{R}_{qq}(q,u) \\
\tfrac{\partial A}{\partial u}(q,u)
= \tfrac{\partial A}{\partial q}(q_0,u_0)\mathbb{R}_{qu}(q,u) + \tfrac{\partial A}{\partial u}(q_0,u_0)
\end{cases}\\
&\Leftrightarrow  \quad
-S'(q)=\tfrac{\partial A}{\partial u}(q,u)^{-1}\tfrac{\partial A}{\partial q}(q,u) 
= \Bigl(\tfrac{\partial A}{\partial q}(q_0,u_0)\mathbb{R}_{qu}(q,u) + \tfrac{\partial A}{\partial u}(q_0,u_0)\Bigr)^{-1}
\tfrac{\partial A}{\partial q}(q_0,u_0)\mathbb{R}_{qq}(q,u) \\ 
&\hspace*{2cm} =
-\tfrac{\partial A}{\partial u}(q_0,u_0)^{-1}\tfrac{\partial A}{\partial q}(q_0,u_0)
\underbrace{\sum_{j=0}^\infty \bigl(\mathbb{R}_{qu}(q,u)
\tfrac{\partial A}{\partial u}(q_0,u_0)^{-1}\tfrac{\partial A}{\partial q}(q_0,u_0)\bigr)^j\mathbb{R}_{qq}(q,u)}_{=:\mathbf{R}(q)} \\[-4ex] 
&\hspace*{2cm}
=-S'(q^0)\mathbf{R}(q)\\
&\Leftrightarrow  \
\forall C\in L(V,Y)\ : \quad 
\mathbf{F}'(q) = \mathbf{F}'(q_0)\mathbf{R}(q)
\end{aligned}
\]  
for arbitrary $(q,u)=(q,S(q))\in U$, $(q_0,u_0)=(q_0,S(q_0))$. Here we have used the fact that allowing for arbitrary observations enforces the second component of the operator $\mathbb{R}$ to be the identity on $u$ and considered the two cases $\uldu=0$ and $\uldq=0$ separately.
\end{proof}

\medskip

\begin{proof}(Lemma~\ref{spectralbounds}) 
If $\nsp(K)^\bot \subseteq \nsp(P)$, then we can estimate
\[
\begin{aligned}
&\|(K^\star K+P^\star P+\alpha I)^{-1}K^\star Kv\|
=\|(K^\star K+P^\star P+\alpha I)^{-1}K^\star K\textup{Proj}_{\nsp(K)^\bot}v\|\\
&=\|(K^\star K+P^\star P+\alpha I)^{-1}(K^\star K+P^\star P)\textup{Proj}_{\nsp(K)^\bot}v\|
\leq \|\textup{Proj}_{\nsp(K)^\bot}v\|\leq \|v\|.
\end{aligned}
\]
Alternatively, if the symmetric operators $(K^\star K)^{1/2}$ and $P^\star P$ commute, then 
\[
\|(K^\star K+P^\star P+\alpha I)^{-1}K^\star Kv\|=
\|(K^\star K)^{1/2}(K^\star K+P^\star P+\alpha I)^{-1}(K^\star K)^{1/2}v\|
\]
and for 
$z:=(K^\star K+P^\star P+\alpha I)^{-1}(K^\star K)v$
we have 
on one hand 
$(K^\star K+P^\star P+\alpha I)z =(K^\star K)v$
which is the first order necessary (and by convexity also sufficient) condition for $z$ minimizing
$\|(K^\star K)^{1/2}(\tilde{z}-v)\|^2+\|(P^\star P)^{1/2}\tilde{z}\|^2+\alpha\|\tilde{z}\|^2$ with respect to $\tilde{z}$. Comparing the value of this objective function at $\tilde{z}=z$ with the one at $\tilde{z}=0$, we obtain
\[
\|(K^\star K)^{1/2}(z-v)\|^2+\|(P^\star P)^{1/2}z\|^2+\alpha\|z\|^2\leq \|(K^\star K)^{1/2}v\|^2
\]
and hence $\|(K^\star K)^{1/2}z\|\leq\|(K^\star K)^{1/2}v\|$, that is, by definition of $z$, 
\[
\|(K^\star K)^{1/2}(K^\star K+P^\star P+\alpha I)^{-1}(K^\star K)^{1/2} (K^\star K)^{1/2}v\|
\leq 2\|(K^\star K)^{1/2}v\|
\]
for any $v\in X$. Extending this estimate by density from the range of $(K^\star K)^{1/2}$ to $\nsp((K^\star K)^{1/2})^\bot=\nsp(K)^\bot$, we arrive at 
\[
\begin{aligned}
&\|(K^\star K)^{1/2}(K^\star K+P^\star P+\alpha I)^{-1}(K^\star K)^{1/2} w\|\\
&=\|(K^\star K)^{1/2}(K^\star K+P^\star P+\alpha I)^{-1}(K^\star K)^{1/2} \textup{Proj}_{\nsp((K^\star K)^{1/2})^\bot}w\|\\
&\leq 2\|\textup{Proj}_{\nsp((K^\star K)^{1/2})^\bot}w\|
\leq 2\|w\| \quad \forall w\in X,
\end{aligned}
\]
which yields
$\|(K^\star K)^{1/2}(K^\star K+P^\star P+\alpha I)^{-1}(K^\star K)^{1/2}\|\leq 2.$

In both cases of \eqref{NKNP} we get
\[
\begin{aligned}
&\|K(K^\star K+P^\star P+\alpha I)^{-1}v\|^2
=\langle(K^\star K+P^\star P+\alpha I)^{-1}v,K^\star K(K^\star K+P^\star P+\alpha I)^{-1}v
\rangle
\leq \frac{C}{\alpha}\|v\|^2.
\end{aligned}
\]
\end{proof}

%% file: AppendixNewton.tex
\section{Two further Newton type methods} \label{Appendix:Newton}


We return to the variational approach \eqref{var}, \eqref{var_eta}, and define an alternative regularized Newton type method for \eqref{FP} by linearizing under the norm in \eqref{var}.

Assuming that $r$ is G\^{a}teaux differentiable, $U$ is weakly closed (e.g., $U$ is closed and convex) and $\mathcal{P}$ is weakly lower semicontinuous, the iterates $(\hat{r}_{n+1}^\delta,x_{n+1}^\delta)\in\widetilde{X}\times U$ are well-defined by
\begin{equation}\label{Newton}
\begin{aligned}
&(\hat{r}_{n+1}^\delta,x_{n+1}^\delta)\in \mbox{argmin}_{(\hat{r},x)\in\widetilde{X}\times U} J_{n}^\delta(\hat{r},x)\\
&\mbox{where } J_{n}^\delta(\hat{r},x)
:=\|K\hat{r}+F(x_0)-y^\delta\|_Y^p+\alpha_n
\mathcal{R}(\hat{r})\\
&\hspace*{4cm}+\beta_n\|r(x_n^\delta)+r'(x_n^\delta)(x-x_n^\delta)-\hat{r}\|_{\widetilde{X}}^b + \mathcal{P}(x)
\end{aligned}
\end{equation}
with some $p,b\in[1,\infty)$, $\alpha_n,\beta_n>0$. 
As above we can avoid imposing conditions for existence of a minimizer and computing it exactly by adding a tolerance $\eta_n>0$ that tends to zero as $n\to\infty$.

Comparison with $(r(x^\dagger),x^\dagger)\in \widetilde{X}\times U$ and using \eqref{delta}, \eqref{rangeinvar_diff} and $\mathcal{P}(x^\dagger)=0$ then yields 
\begin{equation}\label{est_min_Newton}
\begin{aligned}
&\|K(\hat{r}_{n+1}^\delta-r(x^\dagger))+y-y^\delta\|_Y^p
+\alpha_n\Bigl(\mathcal{R}(\hat{r}_{n+1}^\delta)-\mathcal{R}(r(x^\dagger))\Bigr)\\
&+\beta_n\Bigl(\|r'(x_n^\delta)(x_{n+1}^\delta-x^\dagger)+\textup{res}_{\textup{Taylor}}+r(x^\dagger)-\hat{r}_{n+1}^\delta\|_{\widetilde{X}}^b - \|\textup{res}_{\textup{Taylor}}\|_{\widetilde{X}}^b\Bigr) 
+ \mathcal{P}(x_{n+1}^\delta)\\
&\leq\delta^p+\eta_n,
\end{aligned}
\end{equation}
where we have used
\[
\begin{aligned}
\textup{res}_{\textup{Taylor}}:=&r(x_n^\delta)-r(x^\dagger)+r'(x_n^\delta)(x^\dagger-x_n^\delta),\\ 
\textup{res}_{\textup{Newton}}:=&r(x_n^\delta)-r(x^\dagger)+r'(x_n^\delta)(x_{n+1}^\delta-x_n^\delta)
=r'(x_n^\delta)(x_{n+1}^\delta-x^\dagger)+\textup{res}_{\textup{Taylor}}.
\end{aligned}
\]
%
However, the connection between $\hat{r}_{n+1}^\delta$ and $r(x^\dagger)$ by \eqref{est_min_Newton} 
is likely too weak 
to imply smallness of the Taylor remainder $\textup{res}_{\textup{Taylor}}$. 
Still, exact penalization, that is, the choice $b=1$, $\beta_n$ sufficiently large, enforces the corresponding constraint 
$r(x_n^\delta)+r'(x_n^\delta)(x_{n+1}^\delta-x_n^\delta)-\hat{r}_{n+1}^\delta=0$
and thus, inserting the resulting expression for $\hat{r}_{n+1}^\delta$ into the data misfit term and the regularization term, the definition of the Newton iterate reduces to 
\begin{equation}\label{Newton_expen}
\begin{aligned}
&x_{n+1}^\delta\in \mbox{argmin}_{x\in\widetilde{X}\times U} J_{n,1}^\delta(x)\\
&\mbox{where } J_{n,1}^\delta(x)
:=\|K(r(x_n^\delta)+r'(x_n^\delta)(x-x_n^\delta))+F(x_0)-y^\delta\|_Y^p\\
&\hspace*{4cm}+\alpha_n\mathcal{R}(r(x_n^\delta)+r'(x_n^\delta)(x-x_n^\delta))
+ \mathcal{P}(x),
\end{aligned}
\end{equation}
where the data misfit term can be rewritten as
$K(r(x_n^\delta)+r'(x_n^\delta)(x-x_n^\delta))+F(x_0)-y^\delta
=Kr'(x_n^\delta)(x-x_n^\delta) +F(x_n^\delta)-y^\delta
=F(x_n^\delta)+F'(x_n^\delta)(x-x_n^\delta)-y^\delta$.
Thus, as compared to \eqref{frozenNewton},  we work with an enhanced misfit term and have more freedom to choose the space $\tilde{X}$.
As compared to the Levenberg-Marquardt Method (or the iteratively regularized Gauss-Newton method), see, e.g. \cite{Hanke97,Hanke2010,KNSbook:2008,Ried05}, we work with a more structure adapted regularization term.

Comparing the value of $J_{n,1}^\delta$ at the approximate minimizer $x_{n+1}^\delta$ with the one at $x^\dagger$ we obtain
\[
\begin{aligned}
&\|K(r'(x_n^\delta)(x_{n+1}^\delta-x^\dagger)+\textup{res}_{\textup{Taylor}})+y-y^\delta\|_Y^p-\|K\textup{res}_{\textup{Taylor}}+y-y^\delta\|_Y^p\\
&+\alpha_n\Bigl(\mathcal{R}(r(x^\dagger)+r'(x_n^\delta)(x_{n+1}^\delta-x^\dagger)+\textup{res}_{\textup{Taylor}})-\mathcal{R}(r(x^\dagger)+\textup{res}_{\textup{Taylor}})\Bigr)
+ \mathcal{P}(x_{n+1}^\delta)
\leq\delta^p.
\end{aligned}
\]
\Margin{R1 6.}
Deriving a contractive recursion on $e_n=x_n^\delta-x^\dagger$ from this would require Taylor remainder estimates of the form
\[
\|\textup{res}_{\textup{Taylor}}\|_X\leq c \|r'(x_{n-1}^\delta)(x_n^\delta-x^\dagger)\|_{\widetilde{X}}, \quad 
\| K\textup{res}_{\textup{Taylor}}\|_Y\leq c \| Kr'(x_{n-1}^\delta) (x_n^\delta-x^\dagger)\|_Y  
\]
for some $c\in(0,1)$, where the latter is an estimate in image space and thus requires an assumption similar to the tangential cone condition \eqref{tangcone}.
Indeed, to the best of the author's knowledge, convergence results of Newton type methods in a general Banach space setting have so far only been achieved under conditions of the type \eqref{tangcone}, but not yet under a range invariance condition on $F'$. We will thus, like in Section~\ref{subsec:frozenNewton}, also here restrict ourselves to a Hilbert space setting with $p=2$, $\mathcal{R}(\hat{r})=\|\hat{r}\|_{\widetilde{X}}^2$, $\mathcal{P}(x)=\|Px\|_Z^2$, in order to be able to apply spectral theory for bounded selfadjoint operators.
The solution to \eqref{Newton_expen} can then be explicitely expressed via the first order optimality condition 
\begin{equation}\label{NewtonHilbert}
\begin{aligned}
x_{n+1}^\delta=
x_n^\delta+&((KR_n)^\star KR_n+P^\star P+\alpha_n R_n^\star R_n)^{-1}\\
&\Bigl((KR_n)^\star (y^\delta-Kr(x_n^\delta)-F(x_0))
-P^\star Px_n^\delta-\alpha_n R_n^\star r(x_n^\delta)\Bigr)
\end{aligned}
\end{equation}
with $R_n:=r'(x_n^\delta)$.
Due to the fact that $Px^\dagger=0$, the error is thus given by  
\[
\begin{aligned}
&x_{n+1}^\delta-x^\dagger\\
&=((KR_n)^\star KR_n+P^\star P+\alpha_n R_n^\star R_n)^{-1}
\Bigl((KR_n)^\star (y^\delta-y)-((KR_n)^\star K+\alpha_nR_n^\star ) \textup{res}_{\textup{Taylor}} 
-\alpha_n R_n^\star r(x^\dagger)\Bigr)\\
&=(R_n^\star (K^\star K+P_n^\star P_n+\alpha_n \textup{id})R_n)^{-1}
\Bigl(R_n^\star K^\star (y^\delta-y)-R_n^\star (K^\star K+\alpha_n \textup{id}) \textup{res}_{\textup{Taylor}} 
-\alpha_n R_n^\star r(x^\dagger)\Bigr)\\
&=R_n^{-1}(K^\star K+P_n^\star P_n+\alpha_n \textup{id})^{-1}
\Bigl(K^\star (y^\delta-y)-(K^\star K+\alpha_n \textup{id}) \textup{res}_{\textup{Taylor}} 
-\alpha_n r(x^\dagger)\Bigr)
\end{aligned}
\]
with $R_\dagger=r'(x^\dagger)$, $P_n:=PR_n^{-1}$. 
From Lemma~\ref{spectralbounds}, assuming 
\begin{equation}\label{rprime_suff}
\begin{aligned}
&r'(x_0)^{-1}\in L(\widetilde{X},\widehat{X})\textup{ and }\\
&\exists L_r>0\, \forall x\in U\, : \ 
\|r'(x^\dagger)-r'(x)\|_{L(\widehat{X},\widetilde{X})}\leq L_r\|x^\dagger-x\|_{\widehat{X}} <1,
\end{aligned}
\end{equation}
with another auxiliary space $\widehat{X}$, 
we obtain
\begin{equation}\label{est_err_Newton}
\|x_{n+1}^\delta-x^\dagger\|_{\widehat{X}}\leq 
\|R_n^{-1}\|\Bigl(\frac{\delta}{\sqrt{\alpha_n}} 
+ \|\textup{res}_{\textup{Taylor}}\|_{\widetilde{X}}
+ \revision{a_n(P_n)}\Bigr)
\end{equation}
\Margin{R2 11.}
with 
\[
\revision{a_n(P_n)}=\alpha_n\|(K^\star K+P_n^\star P_n+\alpha_n \textup{id})^{-1}r(x^\dagger)\|_{\widetilde{X}}, 
\]
provided (cf. Lemma~\ref{spectralbounds})
\begin{equation}\label{NKNPn}
\textup{(a) }\nsp(K)^\bot \subseteq \nsp(P_n)
\textup{ or \ (b) }P_n^\star P_n(K^\star K)^{1/2} = (K^\star K)^{1/2}P_n^\star P_n.
\end{equation}
From this we can conclude the following convergence result.

\begin{theorem}\label{thm:convNewton}
Let $x_0\in U:=\mathcal{B}_\rho^{\widehat{X}}(x^\dagger)$ for some $\rho>0$ sufficiently small.
Assume that $F$ satisfies \eqref{rangeinvar_diff} with $r$ G\^{a}teaux differentiable and satisfying \eqref{rprime_suff}, 
$r(x^\dagger)\in \bigl(\mathcal{N}(K)\cap \mathcal{N}(PR(x^\dagger)^{-1}\bigr)^\bot$, $K\in L(\widetilde{X},Y)$ and
\begin{equation}\label{NKNPx}
\forall x\in U\, : \ \textup{(a) }\mathcal{N}(K)^\bot = \mathcal{N}(P_x)
\textup{ or \ (b) }(P_x^\star P_x)^{1/2}(K^\star K)^{1/2} = (K^\star K)^{1/2}(P_x^\star P_x)^{1/2}
\end{equation}
with  $P_x:=Pr'(x)^{-1}$.
Let the stopping index $n_*=n_*(\delta)$ be chosen according to \eqref{nstar} with $c$ replaced by $\bar{c}:=2L_r\rho<1$, $L_r$ as in \eqref{rprime_suff}.

Then the iterates $(x_n^\delta)_{n\in\{1,\ldots,n_*(\delta)\}}$ are well-defined by \eqref{Newton}, remain in $\mathcal{B}_\rho^{\widehat{X}}(x^\dagger)$ and converge in $\widehat{X}$, $\|x_{n_*(\delta)}^\delta-x^\dagger\|_{\widehat{X}}\to0$ as $\delta\to0$. In the noise free case $\delta=0$, $n_*(\delta)=\infty$ we have $\|x_n-x^\dagger\|_{\widehat{X}}\to0$ as $n\to\infty$.
\end{theorem}
\begin{proof}
We estimate $\revision{a_n(P_n)}$ by 
\[
\begin{aligned}
\revision{a_n(P_n)}\leq&\|\alpha_n(K^\star K+P_n^\star P_n+\alpha_n \textup{id})^{-1}(P_\dagger^\star P_\dagger-P_n^\star P_n)
(K^\star K+P_\dagger^\star P_\dagger+\alpha_n \textup{id})^{-1}r(x^\dagger)\|_{\widetilde{X}}
+\revision{a_n(P_\dagger)}
\end{aligned}
\]
with the shorthand notation $P_n=Pr'(x_n^\delta)^{-1}$, $P_\dagger=Pr'(x^\dagger)^{-1}$ and 
\[
\begin{aligned}
&\revision{a_n(P_\dagger)}=\|\alpha_n(K^\star K+P_\dagger^\star P_\dagger+\alpha_n \textup{id})^{-1}r(x^\dagger)\|_{\widetilde{X}} \to0\mbox{ as }n\to\infty\\ 
&P_\dagger^\star P_\dagger-P_n^\star P_n=P_n^\star P_n(r'(x_n)r'(x^\dagger)^{-1}-\textup{id})+(\textup{id}-r'(x_n^\delta)^{-1\,\star}r'(x^\dagger)^\star )P_\dagger^\star P_\dagger.
\end{aligned}
\]
Due to Lemma~\ref{spectralbounds} (with reversed roles of $K$ and $P$) we have 
\[
\begin{aligned}
&\|(K^\star K+P_n^\star P_n+\alpha_n \textup{id})^{-1}P_n^\star P_n\|\leq C, \quad&&
\alpha_n\|(K^\star K+P_n^\star P_n+\alpha_n \textup{id})^{-1}\|\leq 1, \\
&\|P_\dagger^\star P_\dagger(K^\star K+P_\dagger^\star P_\dagger+\alpha_n \textup{id})^{-1}\|\leq C, \quad&&
\alpha_n\|(K^\star K+P_\dagger^\star P_\dagger+\alpha_n \textup{id})^{-1}\|\leq 1. 
\end{aligned}
\]
Moreover, 
\[
\begin{aligned}
&\|r'(x_n^\delta)r'(x^\dagger)^{-1}-\textup{id}\|=\|(r'(x_n^\delta)-r'(x^\dagger))r'(x^\dagger)^{-1}\|
\leq c\|r'(x^\dagger)^{-1}\|\,\|x_n^\delta-x^\dagger\|_X\\
&\|\textup{id}-r'(x_n^\delta)^{-1\,\star}r'(x^\dagger)^\star \|=\|(r'(x_n^\delta)-r'(x^\dagger))r'(x_n^\delta)^{-1}\|
\leq c\|r'(x_n^\delta)^{-1}\|\,\|x_n^\delta-x^\dagger\|_X,
\end{aligned}
\]
and due \eqref{rprime_suff}, the Taylor remainder can be estimated by means of 
\begin{equation}\label{Taylor}
\|r(x)-r(x^\dagger)+r'(x)(x^\dagger-x)\|_{\widetilde{X}} 
\leq \frac{L_r}{2} \|x^\dagger-x\|_{\widehat{X}}^2,
\end{equation} 
which yields an estimate of the form
\[
\|x_{n+1}^\delta-x^\dagger\|_{\widehat{X}}\leq 
C\Bigl(\frac{\delta}{\sqrt{\alpha_n}} + \|x_n^\delta-x^\dagger\|_{\widehat{X}}^2
+ \revision{a_n(P_\dagger)}\Bigr).
\]
Thus, by choosing $\rho$ sufficiently small, we can conclude from $x_n\in \mathcal{B}_\rho(x^\dagger)$ and \eqref{est_err_Newton} that for some $C>0$, $\bar{c}\in(0,1)$ independent of $n$,
\begin{equation}\label{errest_Newton}
\|x_{n+1}^\delta-x^\dagger\|_{\widehat{X}}\leq 
C\frac{\delta}{\sqrt{\alpha_n}} + \bar{c}\|x_n^\delta-x^\dagger\|_{\widehat{X}}
+ C \revision{a_n(P_\dagger)}.
\end{equation}
Thus $x_{n+1}\in \mathcal{B}_\rho(x^\dagger)$ and as in the proof of Theorem~\ref{thm:convfrozenNewton}, the assertions follow.
\end{proof}

To arrive at a frozen Newton method in place of \eqref{NewtonHilbert}, we replace $r'(x_n)$ by $r'(x_0)$ in the above 
\begin{equation}\label{NewtonHilbert_frozen}
\begin{aligned}
x_{n+1}^\delta=
x_n^\delta+&
((KR_0)^\star KR_0+P^\star P+\alpha_n R_0^\star R_0)^{-1}\\
&\Bigl((KR_0)^\star (y^\delta-Kr(x_n^\delta)-F(x_0))
-P^\star Px_n^\delta-\alpha_n R_0^\star r(x_n^\delta)\Bigr)
\end{aligned}
\end{equation}
with $R_0=r'(x_0)$, for which we obtain the estimate
\begin{equation}\label{est_err_Newton_frozen}
\|x_{n+1}^\delta-x^\dagger\|_{\widehat{X}}\leq 
\|r'(x_0)^{-1}\|_{L(\widetilde{X},\widehat{X})}\Bigl(
\frac{\delta}{\sqrt{\alpha_n}} + c_0\|x_n^\delta-x^\dagger\|_{\widehat{X}}
+ \revision{a_n(PR_0^{-1})}\Bigr)
\end{equation}
with 
\[
\revision{a_n(PR_0^{-1})}=\alpha_n\|(K^\star K+(PR_0^{-1})^\star PR_0^{-1}+\alpha_n \textup{id})^{-1}r(x^\dagger)\|_{\widetilde{X}} \to0\mbox{ as }n\to\infty 
\]
provided $r(x^\dagger)\in (\nsp(K)\cap \nsp(PR_0^{-1}))^\bot\subseteq \nsp(K^\star K+(PR_0^{-1})^\star PR_0^{-1})^\bot$ and bounded invertibility as well as a Taylor remainder estimate hold
\begin{equation}\label{Taylor0}
\begin{aligned}
&r'(x_0)^{-1}\in L(\widetilde{X},\widehat{X})\textup{ and }\\
&\exists c_0\in(0,1)\, \forall x\in U\, : \ \|r(x)-r(x^\dagger)+r'(x_0)(x^\dagger-x)\|_{\widetilde{X}}\leq c_0 \|x^\dagger-x\|_{\widehat{X}}.
\end{aligned}
\end{equation}

Similarly to Theorem~\ref{thm:convNewton} we obtain
\begin{theorem}\label{thm:convNewton0}
Let $x_0\in U:=\mathcal{B}_\rho^{\widehat{X}}(x^\dagger)$ for some $\rho>0$ sufficiently small.
Assume that $F$ satisfies \eqref{rangeinvar_diff} with \eqref{Taylor0},  
$r(x^\dagger)\in \bigl(\nsp(K)\cap \nsp(PR(x_0)^{-1})\bigr)^\bot$, $K\in L(\widetilde{X},Y)$, and
\begin{equation}\label{NKNP0}
\textup{(a) }\nsp(K)^\bot \subseteq \nsp(P_0)
\textup{ or \ (b) }P_0^\star P_0(K^\star K)^{1/2} = (K^\star K)^{1/2}P_0^\star P_0
\end{equation}
with  $P_0:=Pr'(x_0)^{-1}$.
Let the stopping index $n_*=n_*(\delta)$ be chosen according to \eqref{nstar} with $c$ replaced by $c_0$ and $c_0$ as in \eqref{Taylor0}.

Then the iterates $(x_n^\delta)_{n\in\{1,\ldots,n_*(\delta)\}}$ are well-defined by \eqref{NewtonHilbert_frozen}, remain in $\mathcal{B}_\rho^{\widehat{X}}(x^\dagger)$ and converge in $X$, $\|x_{n_*(\delta)}^\delta-x^\dagger\|_{\widehat{X}}\to0$ as $\delta\to0$. In the noise free case $\delta=0$, $n_*(\delta)=\infty$ we have $\|x_n-x^\dagger\|_{\widehat{X}}\to0$ as $n\to\infty$.
\end{theorem}

Note that in case of the canonical choice $K=F'(x_0)$ we have $r'(x_0)=\textup{id}$ and therefore \eqref{Taylor0} enforces $\widehat{X}=\widetilde{X}$.

\subsection{Convergence in the setting of Section~\ref{sec:exclass}}
In the reduced setting of Section~\ref{sec:ex_red}, Theorems~\ref{thm:convNewton}, \ref{thm:convNewton0} yield the following two corollaries.

\begin{corollary}[Newton, reduced]\label{cor:convNewton_red}
Let $q_0\in U:=\mathcal{B}_\rho^{\widetilde{Q}_J}(q^\dagger)\subseteq\mathcal{D}(\mathbf{F})$ for some $\rho>0$ sufficiently small, assume that $B:V_J\to L(\widetilde{Q}_J,W^*_J)$ is Lipschitz continuously differentiable, and $D\in L(V_J,W^*_J)$.
Moreover, let 
$r(q^\dagger)\in \bigl(\nsp(\mathbf{F}'(q_0))\cap \nsp(Pr'(q^\dagger)^{-1})\bigr)^\bot$,
\eqref{NKNPx} hold with $K=\mathbf{F}'(q_0)$, and the stopping index $n_*=n_*(\delta)$ be chosen according to \eqref{nstar}.
\\
Then the iterates $(q_n^\delta)_{n\in\{1,\ldots,n_*(\delta)\}}$ are well-defined by \eqref{NewtonHilbert}, \eqref{F_red}, remain in $\mathcal{B}_\rho^{\widetilde{Q}_J}(q^\dagger)$ and converge in $\widetilde{Q}_J$, $\|q_{n_*(\delta)}^\delta-q^\dagger\|_{\widetilde{Q}_J}\to0$ as $\delta\to0$. In the noise free case $\delta=0$, $n_*(\delta)=\infty$ we have $\|q_n-q^\dagger\|_{\widetilde{Q}_J}\to0$ as $n\to\infty$.
\end{corollary}

\begin{proof}
To establish \eqref{rprime_suff}, note that $K=\mathbf{F}(q_0)$, thus $r'(q_0)=\textup{id}$ and we have set $\widehat{X}=\widetilde{X}=\widetilde{Q}_J$. Moreover, with 
$\underline{dr}:=\bigl({r^{(0)}_\textup{diff}}'(q^\dagger,S(q^\dagger);q_0,S(q_0))
-{r^{(0)}_\textup{diff}}'(q,S(q);q_0,S(q_0))\bigr)\uldq$ we have 
\[
\|(r'(q^\dagger)-r'(q))\uldq\|_{\widetilde{Q}_J}=
\|\underline{dr}\|_{\widetilde{Q}^{(0)}_J}
\sim \|B_0\underline{dr}\|_{W^*_J}\,,
\]
due to the assumed isomorphism property of $B_0$.
With $u=S(q)$, $u^\dagger=S(q^\dagger)$, $v=S'(q)\uldq$, $v^\dagger=S'(q^\dagger)\uldq$ we can rewrite 
\[
\begin{aligned}
B_0\underline{dr}=&
[B'(u^\dagger)v^\dagger]q^\dagger-[B'(u_0)v^\dagger]q_0+B(u^\dagger)\uldq
-([B'(u)v]q-[B'(u_0)v]q_0+B(u)\uldq)\\
=&
[(B'(u^\dagger)-B'(u))v^\dagger]q^\dagger
+[B'(u)(v^\dagger-v)]q^\dagger
+[B'(u)v](q^\dagger-q)\\
&-[B'(u_0)(v^\dagger-v)]q_0
+[B(u^\dagger)-B(u)]\uldq
\end{aligned}
\]
where as in \eqref{Sqdiff}, \eqref{Fprime_red}
\[
\begin{aligned}
&u^\dagger-u = -\Bigl(D+[\overline{B'}\cdot] q\Bigr)^{-1}B(u^\dagger)(q^\dagger-q), \\
&v^\dagger =-\Bigl(D+[B'(u^\dagger)\cdot] q^\dagger\Bigr)^{-1}B(u^\dagger)\uldq, \quad
v =-\Bigl(D+[B'(u)\cdot] q\Bigr)^{-1}B(u)\uldq
\end{aligned}
\]
with $[\overline{B'}\uldu]q=\int_0^1[B'(u+\theta(u^\dagger-u))\uldu]q\, d\theta$.
Thus using our assumption of $B$ being Lipschitz continuously differentiable as an operator $V_J\to L(\widetilde{Q}_J,W^*_J)$, we can estimate
$\|B_0\underline{dr}\|_{W^*_J}\leq L_r\|q^\dagger-q\|_{\widetilde{Q}_J}$ for some constant $L_r$.
\end{proof}
\begin{corollary}[alternative frozen Newton, reduced]\label{cor:convNewton0_red}
Let $q_0\in U:=\mathcal{B}_\rho^{\widetilde{Q}_J}(q^\dagger)\subseteq\mathcal{D}(\mathbf{F})$ for some $\rho>0$ sufficiently small, assume that $B:V_J\to L(\widetilde{Q}_J,W^*_J)$ is continuously differentiable, and $D\in L(V_J,W^*_J)$.
Moreover, let 
$r(q^\dagger)\in \bigl(\nsp(\mathbf{F}'(q_0))\cap \nsp(P)\bigr)^\bot$,
\eqref{NKNP0} hold with $K=\mathbf{F}'(q_0)$, $P_0:=P$ and the stopping index $n_*=n_*(\delta)$ be chosen according to \eqref{nstar}.
\\
Then the iterates $(q_n^\delta)_{n\in\{1,\ldots,n_*(\delta)\}}$ are well-defined by \eqref{NewtonHilbert_frozen}, \eqref{F_red}, remain in $\mathcal{B}_\rho^{\widetilde{Q}_J}(q^\dagger)$ and converge in $\widetilde{Q}_J$, $\|q_{n_*(\delta)}^\delta-q^\dagger\|_{\widetilde{Q}_J}\to0$ as $\delta\to0$. In the noise free case $\delta=0$, $n_*(\delta)=\infty$ we have $\|q_n-q^\dagger\|_{\widetilde{Q}_J}\to0$ as $n\to\infty$.
\end{corollary}
\begin{proof}
Also here, $r'(x_0)=\textup{id}$ and we only need to estimate the Taylor remainder \eqref{Taylor0} for the nonlinear part $r^{(0)}_\textup{diff}$ of $r$, which due to the fact that 
\[
\frac{d}{dq}r^{(0)}_\textup{diff}(q_0,S(q_0);q_0,u_0)\uldq
=\frac{\partial r^{(0)}_\textup{diff}}{\partial q}(q_0,S(q_0);q_0,u_0)\uldq
+\frac{\partial r^{(0)}_\textup{diff}}{\partial u}(q_0,S(q_0);q_0,u_0)S'(q_0)\uldq
\]
where
\begin{equation}\label{rdiffprime0}
\begin{aligned}
&\frac{\partial r^{(0)}_\textup{diff}}{\partial q}(q_0,S(q_0);q_0,u_0)\uldq 
=\Bigl[B_0^{-1}\Bigl([B(u)-B(u_0)](q-q_0)\Bigr)\Bigr]_{q=q_0,u=u_0}=0\\
&\frac{\partial r^{(0)}_\textup{diff}}{\partial q}(q_0,S(q_0);q_0,u_0)\uldq 
\\&
=\Bigl[B_0^{-1}\Bigl([B'(u)\uldu-B'(u_0)\uldu]q_0
+[B'(u)\uldu](q-q_0)\Bigr)\Bigr]_{q=q_0,u=u_0}=0
\end{aligned}
\end{equation}
is just equal to \eqref{est_rdiff}.
Using \eqref{udiff}, the assumed isomorphism property of $B_0$ and the assumed continuous differentiability of $B$ as an operator $V_J\to L(\widetilde{Q}_J,W^*_J)$ in \eqref{est_rdiff} yields the desired estimate.
\end{proof}
\medskip

In the all-at-once setting of Section~\ref{sec:ex_aao} we obtain the following.
\begin{corollary}[Newton, all-at-once]\label{cor:convNewton_aao}
Let $(q_0,u_0)\in U:=\mathcal{B}_\rho^{\widetilde{Q}_J\times V_J}(q^\dagger,u^\dagger)\subseteq \widetilde{Q}_J\times V_J$
for some $\rho>0$ sufficiently small, assume that $B:V_J\to L(\widetilde{Q}_J,W^*_J)$ is Lipschitz continuously differentiable, and $D\in L(V_J,W^*_J)$.
Moreover, let 
$r(q^\dagger,u^\dagger)\in \bigl(\nsp(\mathbb{F}'(q_0,u_0))\cap \nsp(Pr'(q^\dagger,u^\dagger)^{-1})\bigr)^\bot$,
\eqref{NKNPx} hold with $K=\mathbb{F}'(q_0,u_0)$, \eqref{Fprime_aao} and the stopping index $n_*=n_*(\delta)$ be chosen according to \eqref{nstar}.
\\
Then the iterates $(q_n^\delta,u_n^\delta)_{n\in\{1,\ldots,n_*(\delta)\}}$ are well-defined by \eqref{NewtonHilbert}, \eqref{F_aao}, remain in $\mathcal{B}_\rho^{\widetilde{Q}_J\times V_J}(q^\dagger,u^\dagger)$ and converge in $\widetilde{Q}_J\times V_J$, $\|q_{n_*(\delta)}^\delta-q^\dagger\|_{\widetilde{Q}_J}\to0$, $\|u_{n_*(\delta)}^\delta-u^\dagger\|_{V_J}\to0$ as $\delta\to0$. In the noise free case $\delta=0$, $n_*(\delta)=\infty$ we have $\|q_n-q^\dagger\|_{\widetilde{Q}_J}\to0$, $\|u_n-u^\dagger\|_{V_J}\to0$ as $n\to\infty$.
\end{corollary}
\begin{proof}
We obtain \eqref{rprime_suff} by estimating
\[
\begin{aligned}
&\|(r'(q^\dagger,u^\dagger)-r'(q,u))(\uldq,\uldu)\|_{\widetilde{Q}_J\times V_J}\\
&=\|\bigl({r^{(0)}_\textup{diff}}'(q^\dagger,S(q^\dagger);q_0,S(q_0))
-{r^{(0)}_\textup{diff}}'(q,S(q);q_0,S(q_0))\bigr)(\uldq,\uldu)
\|_{\widetilde{Q}^{(0)}_J}\\
&\leq\|B_0^{-1}\|_{W^*,\widetilde{Q}^{(0)}_J}
\|[(B'(u^\dagger)-B'(u))\uldu]q^\dagger+[B'(u)\uldu](q^\dagger-q)
+[B(u^\dagger)-B(u)]\uldq\|_{W^*_J}.
\end{aligned}
\]
\end{proof}
\begin{corollary}[alternative frozen Newton, all-at-once]\label{cor:convNewton0_aao}
Let $(q_0,u_0)\in U:=$\\ $\mathcal{B}_\rho^{\widetilde{Q}_J\times V_J}(q^\dagger,u^\dagger)\subseteq \widetilde{Q}_J\times V_J$
for some $\rho>0$ sufficiently small, assume that $B:V_J\to L(\widetilde{Q}_J,W^*_J)$ is continuously differentiable, and $D\in L(V_J,W^*_J)$.
Moreover, let 
$r(q^\dagger,u^\dagger)\in \bigl(\nsp(\mathbb{F}'(q_0,u_0))\cap \nsp(Pr'(q_0,u_0)^{-1})\bigr)^\bot$,
\eqref{NKNP0} hold with $K=\mathbb{F}'(q_0,u_0)$, $P_0:=Pr'(q_0,u_0)^{-1}$ and the stopping index $n_*=n_*(\delta)$ be chosen according to \eqref{nstar}.
\\
Then the iterates $(q_n^\delta,u_n^\delta)_{n\in\{1,\ldots,n_*(\delta)\}}$ are well-defined by \eqref{NewtonHilbert_frozen}, \eqref{F_aao}, remain in $\mathcal{B}_\rho^{\widetilde{Q}_J\times V_J}(q^\dagger,u^\dagger)$ and converge in $\widetilde{Q}_J\times V_J$, $\|q_{n_*(\delta)}^\delta-q^\dagger\|_{\widetilde{Q}_J}\to0$, $\|u_{n_*(\delta)}^\delta-u^\dagger\|_{V_J}\to0$ as $\delta\to0$. In the noise free case $\delta=0$, $n_*(\delta)=\infty$ we have $\|q_n-q^\dagger\|_{\widetilde{Q}_J}\to0$, $\|u_n-u^\dagger\|_{V_J}\to0$ as $n\to\infty$.
\end{corollary}
\begin{proof}
Again the crucial estimate to establish \eqref{Taylor0} is the one on the Taylor remainder of $\textup{res}_{0\textup{Taylor}}(r^{(0)}_\textup{diff})$ and due to \eqref{rdiffprime0} equals \eqref{est_rdiff}.
\end{proof}

%% file: AppendixLinUniqueness.tex
\section{Linearized uniqueness proofs} \label{Appendix:LinUniqueness}

\subsection{Nullspace condition \eqref{NKNP_ex} for the elliptic case in Example~\ref{ex:pot_revisited}}\label{sec_LinUni_pot_rev_ell}

We use the fact that $\mathbf{F}'(q_0)\uldq=0$ with $\uldu^j=(S^j)'(q_0)\uldq$, $u_0^j=S^j(q_0)$ reads as
\[
\int_\Omega \Bigl(\nabla \uldu^j\cdot\nabla w+q_0\uldu^jw+\uldq u_0^j w\Bigr)\, dx
=0 \,, \forall w\in W:=H^1(\Omega), \quad
\textup{tr}_{\partial\Omega}\uldu^j=0 \qquad j\in\mathbb{N}.
\]
From this, using integration by parts, we conclude that 
\[
\forall j\in\mathbb{N}\,\forall w\in M_{q_0}:=\{v\in H^2(\Omega)\, : \, -\Delta v+q_0\cdot v=0\}\, : \int_\Omega \uldq u^j w\, dx
=0.
\]
(Note that no boundary conditions are prescribed on $v\in M_{q_0}$.)
Since the boundary values $\varphi^j$ of $u_0^j$ run over a basis of $H^{-1/2}(\partial\Omega)$, this is equivalent to 
\[
\forall w_1,\,w_2\,\in M_{q_0}: \int_\Omega \uldq\, w_1\, w_2\, dx
=0,
\]
which means that $\uldq$ is orthogonal to the set of products of functions in $M_{q_0}$. This set has been shown to be dense in $L^2(\Omega)$, see, e.g., \cite[Theorem 6.25]{Kirsch:2011}, which implies $\uldq=0$.

\subsection{Nullspace condition \eqref{NKNP_ex} for the transient case in Example~\ref{ex:pot_revisited}}\label{sec_LinUni_pot_rev_trans}

$\mathbb{F}'(q_0,u_0)(\uldq,\uldu)=0$ is equivalent to $(D+q_0\cdot)\uldu+\uldq\cdot(u_0+\bar{h})=0$ and $C\uldu=0$, where we assume $(u_0+\bar{h})(x,t)=\phi(x)\psi(t)$.
Taking the $L^2_w$ inner product with $\varphi_i^k$ as well as the Laplace transform with respect to time, and defining the complex functions $\omega_{\lambda_i}(z):=\sum_{m=1}^M z^m\widehat{\mathcal{M}}_m(z)+\sum_{n=0}^N z^n\widehat{\mathcal{N}}_n(z)*\lambda_i^{\beta_n}$ yields
\[
\begin{aligned}
&\omega_{\lambda_i}(z)\,\widehat{\uldu}_i^k (z)= -
\langle \uldq\cdot\phi,\varphi_i^k\rangle_{L^2_w(\Omega)}\widehat{\psi}(z), \ \ i\in\mathbb{N}, \ k\in K^i, 
\quad \textup{ and }\
\sum_{j\in\mathbb{N}}\sum_{k\in K^j} \widehat{\uldu}_j^k(z) \, C \varphi_j^k = 0
\end{aligned}
\]
for $\uldu_i^k:= \langle\uldu,\varphi_i^k\rangle_{L^2_w(\Omega)}$, and all $z\in\mathbb{C}$, which can be rewritten as a system
\begin{equation}\label{Fprimezero_transient}
\sum_{j\in\mathbb{N}}\frac{\widehat{\psi}(z)}{\omega_{\lambda_j}(z)}
\sum_{k\in K^j} \,a_j^k\, C \varphi_j^k \, = 0
\end{equation}
(valid for all $z\in\mathbb{C}$ except for the roots of $\omega_{\lambda_j}$, $j\in\mathbb{N}$)
for the coefficients $a_i^k:=\langle \uldq\cdot\phi,\varphi_i^k\rangle_{L^2_w(\Omega)}$.
In order to first of all disentangle the sum over $j\in\mathbb{N}$, we assume that the poles of $\frac{1}{\omega_{\lambda}}$ differ for different $\lambda$, that is
\begin{equation}\label{polesdiffer}
\forall i\in\mathbb{N}\, \exists p_i\in \mathbb{C}\, : \ \omega_{\lambda_i}(p_i)=0 \text{ and } \forall j\in\mathbb{N}\setminus \{i\}\, : \ \omega_{\lambda_j}(p_i)\not=0.
\end{equation}
Taking the residues of \eqref{Fprimezero_transient} we can thus use the fact that 
$\textup{Res}(\frac{1}{\omega_{\lambda_i}};p_i)\not=0$ and 
$\textup{Res}(\frac{1}{\omega_{\lambda_j}};p_i)=0$ for $j\not=i$, which yields 
\begin{equation}\label{Fprimezero_transient_1}
\forall i\in\mathbb{N}\, : \widehat{\psi}(p_i) \sum_{k\in K^i}  \,a_i^k\, C \varphi_i^k \, = 0
\end{equation}
We additionally suppose that $\psi$ was chosen such that $\widehat{\psi}(p_i)\not=0$ for all $i\in\mathbb{N}$ and make the linear independence assumption on the individual eigenspaces
\[
\forall i\in\mathbb{N}\, \Bigl(\forall \vec{b}\in \mathbb{R}^{K^i}\, : \ 
\sum_{k\in K^i}  \,b^k\, C \varphi_i^k \, = 0\ \Rightarrow \ \vec{b}=0\Bigr).
\]
(In the context of $C:=\textup{tr}_\Gamma$ this means that the eigenfunctions should not lose their linear independence when restricted to the observation set $\Gamma$; see \cite[Remark 4.1]{nonlinearity_imaging_fracWest} 
for some examples and comments.)
This allows to conclude that for all $i\in\mathbb{N}$, $k\in K^i$, we have $a_i^k=0$ and thus $\uldq\cdot\phi=\sum_{j\in\mathbb{N}}\sum_{k\in K^j} \,a_j^k\, C \varphi_j^k \, = 0$. Having chosen $\phi$ nonzero almost everywhere, we can conclude $\uldq=0$.
\\
Condition \eqref{polesdiffer} can be easily verified in the parabolic and hyperbolic cases, but also in certain fractional damping models, cf., e.g, \cite{nonlinearity_imaging_fracWest}.

\subsection{Nullspace condition \eqref{NKNP_ex} for  Example~\ref{ex:diffabs}}\label{sec_LinUni_diffabs}
Here, $\mathbf{F}'(q_0)\uldq=0$ with $\uldu^{\lambda,n}=(S^{\lambda,n})'(q_0)\uldq$, $u_0^{\lambda,n}=S^{\lambda,n}(q_0)$ reads as
\[
\begin{aligned}
&\int_\Omega \Bigl((a_0\nabla \uldu^{\lambda,n}+\ulda \nabla u^{\lambda,n})\cdot\nabla w+(c_0\,\uldu^{\lambda,n}+\uldc\, u_0^{\lambda,n}) w\Bigr)\, dx
=0 \,, \forall w\in H^1(\Omega), 
\textup{ and }\textup{tr}_{\partial\Omega}\uldu^{\lambda,n}=0 \\ 
&n\in\mathbb{N}, \quad \lambda\geq0.
\end{aligned}
\]
From this, integrating by parts and using completeness of $(\varphi^n)_{n\in\mathbb{N}}$ in $H^{-1/2}(\partial\Omega)$, we obtain the orthogonality relation
\begin{equation}\label{orth_diffabs}
\forall \lambda\geq0 \ \forall w_1,\,w_2\,\in M_{q_0}^\lambda: \ \int_\Omega \ulda\nabla w_1\cdot\nabla w_2 + \uldc w_1 w_2\, dx =0.
\end{equation}
with $M_{q_0}^\lambda:=\{v\in H^2(\Omega)\, : \, -\nabla\cdot( a_0\nabla v)+(c_0-\lambda)\cdot v=0\}$.
We restrict ourselves to the simple choice $a_0\equiv\text{const.}>0$, $c_0\equiv0$ so that, since $\lambda$ runs over all nonnegative real numbers, without loss of generality  $M_{q_0}^\lambda$ can be replaced by $M_{1,0}^\lambda:=\{v\in H^2(\Omega)\, : \, -\Delta v=\lambda v\}$ and similarly to the proof of \cite[Theorem 6.25]{Kirsch:2011} use plane wave elements
$w_{\lambda,k}(x):=e^{\imath\sqrt{\lambda}\hat{\theta}_k\cdot x}$, $k\in\{1,2\}$ of $M_{1,0}^\lambda$, with $\hat{\theta}_k$ being appropriately selected unit vectors in $\mathbb{R}^d$.\\
For $y\in\mathbb{R}^d$ arbitrary but fixed we choose 
$\lambda:=|y|^2$, $\theta_1:=\frac{y}{|y|}$, $\theta_2=0$ so that inserting $w_{\lambda,1},\,w_{\lambda,2}$ in \eqref{orth_diffabs} yields $ \int_\Omega \uldc \,e^{\imath y\cdot x}\, dx =0$. Since $y\in\mathbb{R}^d$ was arbitrary, this implies $\uldc=0$.\\
Again, fixing an arbitrary $y\in\mathbb{R}^d\setminus\{0\}$, we now choose $y^\bot$ orthogonal to $y$ and normalized $|y^\bot|=1$, $\epsilon\in(0,|y|^2)$ and set $\lambda:=|y|^2-\epsilon$, $\theta_{1/2}=y\pm\epsilon y^\bot$. It is readily checked that inserting the corresponding $w_{\lambda,1},\,w_{\lambda,2}$ in \eqref{orth_diffabs} yields $ \int_\Omega \ulda \,e^{\imath y\cdot x}\, dx =0$ so that by varying $y\in\mathbb{R}^d\setminus\{0\}$we obtain $\ulda=0$.